\documentclass[a4paper,leqno,11pt]{amsart}

\usepackage{amsfonts,amssymb,verbatim,amsmath,amsthm,latexsym,textcomp,amscd}
\usepackage{latexsym,amsfonts,amssymb,epsfig,verbatim}
\usepackage{amsmath,amsthm,amssymb,latexsym,graphics,textcomp}
\usepackage{mathtools}
\usepackage{paralist}
\usepackage{graphicx}
\usepackage{color}
\usepackage{url}
\usepackage{enumerate}
\usepackage[mathscr]{euscript}
\usepackage{tikz-cd}
\usetikzlibrary{shapes.geometric}
\usepackage{dsfont}
\usetikzlibrary{matrix}
\usepackage{hyperref}
\usepackage{multirow}

\input xy

\xyoption{all}

\theoremstyle{definition}
\newtheorem{theorem}{Theorem}[section]
\newtheorem{prop}[theorem]{Proposition}

\newtheorem{cor}[theorem]{Corollary}
\newtheorem{defn}[theorem]{Definition}

\newtheorem{rmk}[theorem]{Remark}

\newtheorem{exam}[theorem]{Example}

\newtheorem{notation}[theorem]{Notation}
\newtheorem{subsec}[theorem]{}

\theoremstyle{plain}
\newtheorem*{thma}{Theorem A}
\newtheorem*{thmb}{Theorem B}
\newtheorem*{thmc}{Theorem C}
\newtheorem*{thmd}{Theorem D}

\theoremstyle{remark}
{\swapnumbers
    }
   
\newenvironment{myeq}[1][]
{\stepcounter{theorem}\begin{equation}\tag{\thetheorem}{#1}}
{\end{equation}}

\newenvironment{mysubsection}[2][]
{\begin{subsec}\begin{upshape}\begin{bfseries}{#2.}
			\end{bfseries}{#1}}
		{\end{upshape}\end{subsec}}

\newcommand{\C}{{\mathbb C}}

\newcommand{\Z}{{\mathbb{Z}}}

\newcommand{\Q}{{\mathbb Q}}

\def\stackbelow#1#2{\underset{\displaystyle\overset{\displaystyle\shortparallel}{#2}}{#1}}

\def\stackupeq#1#2{\overset{\displaystyle\underset{\displaystyle{\rotatebox{90}{$\cong$}}}{#2}}{#1}}
\def\stackbeloweq#1#2{\underset{\displaystyle\overset{\displaystyle{\rotatebox{90}{$\cong$}}}{#2}}{#1}}

\newcommand\CC{{\mathbb C}}
\newcommand\DD{{\mathcal D}}

\newcommand\FF{{\mathcal F}}

\newcommand\HH{{\mathbb H}}

\newcommand\LL{{\mathcal L}}
\newcommand\MM{{\mathcal M}}

\newcommand\PP{{\mathcal P}}

\newcommand\PMF{{\PP\kern-2pt\MM\FF}}

\newcommand\PML{{\PP\kern-2pt\MM\LL}}

\newcommand{\fsubd}{\mathrel{{\scriptstyle\searrow}\kern-1ex^d\kern0.5ex}}
\newcommand{\bsubd}{\mathrel{{\scriptstyle\swarrow}\kern-1.6ex^d\kern0.8ex}}
\newcommand{\fsubeq}{\mathrel{\raise-.7ex\hbox{$\overset{\searrow}{=}$}}}
\newcommand{\bsubeq}{\mathrel{\raise-.7ex\hbox{$\overset{\swarrow}{=}$}}}

\newcommand{\tsh}[1]{\left\{\kern-.9ex\left\{#1\right\}\kern-.9ex\right\}}

\newcommand{\Rank}{\mbox{Rank}}

\makeatletter
\@namedef{subjclassname@2020}{%
  \textup{2020} Mathematics Subject Classification}
\makeatother

\title[$SU(2)$-bundles over highly connected $8$-manifolds]{$SU(2)$-bundles over highly connected $8$-manifolds}

\author[S. Basu]{Samik Basu}
\address{Stat Math Unit, Indian Statistical Institute Kolkata 700108, India}
\email{samik.basu2@gmail.com}

\author[A. K. Ghosh]{Aloke Kr Ghosh}
\address{Stat Math Unit, Indian Statistical Institute Kolkata 700108, India}
\email{alokekrghosh005@gmail.com}

\author[S. Sau]{Subhankar Sau}
\address{Stat Math Unit, Indian Statistical Institute Kolkata 700108, India}
\email{subhankarsau18@gmail.com}


\begin{document}

\begin{abstract}
In this paper, we analyze the possible homotopy types of the total space of a principal $SU(2)$-bundle over a $3$-connected $8$-dimensional Poincar\'{e} duality complex. Along the way, we also classify the $3$-connected $11$-dimensional complexes $E$ formed from a wedge of $S^4$ and $S^7$ by attaching a $11$-cell. 
\end{abstract}

\maketitle

\section{Introduction}
This paper explores $SU(2)$-bundles over $8$-manifolds, aiming for results akin to those about circle bundles over $4$-manifolds \cite{DuLi05,BaBa15}. In the case of simply connected $4$-manifolds, the results are established by leveraging the classification of simply connected $5$-manifolds achieved by Smale \cite{Sma62} and Barden \cite{Bar65}. 

A circle bundle $S^1 \to X \to M$ over a simply connected $4$-manifold $M$ is classified by $\alpha\in H^2(M)$, the total space $X(\alpha)$ is simply connected if $\alpha$ is primitive, and there are only two possibilities of $X(\alpha)$ via the classification of simply connected $5$-manifolds. Explicitly, we have \cite[Theorem 2]{DuLi05} 
\begin{enumerate}
\item  For every simply connected $4$-manifold $M$, there is a circle bundle $\alpha$, such that $X(\alpha)$ is homotopy equivalent to a connected sum of $S^2 \times S^3$. If $M$ is spin, among primitive $\alpha$, this is the only possibility. 
\item 
  For a simply connected $4$-manifold $M$ which is not spin and  a circle bundle $\alpha$ over it,  $X(\alpha)$ is either homotopy equivalent to a connected sum of $S^2 \times S^3$, or to a  connected sum of $S^2 \times S^3$ and another manifold $B$. The manifold $B$ is (up to diffeomorphism unique) a non-spin simply connected $5$-manifold whose homology is torsion-free, and $\Rank(H_2(B))=1$. 
\end{enumerate}
The results of Smale and Barden are geometric in nature, and do not generalize easily to higher dimensions. Using homotopy theoretic methods, it was possible to construct sphere fibrations \cite{BaGh2023} over highly connected Poincar\'{e}-duality complexes possibly by inverting a few primes or in high enough rank. Among these sphere fibrations, the only case where they could be principal bundles was in dimension $8$, and the question whether they may be realized as such was left unresolved.  

In this paper, we consider principal $SU(2)$-bundles, noting that $SU(2)=S^3$ is the only case apart from the circle where the sphere is a Lie group. The base space of the $SU(2)$-bundle which is appropriate for making a similar analysis is a highly connected $8$-manifold. More precisely, we consider Poincar\'{e} duality complexes $M$ ($8$-dimensional) that are $3$-connected. These are obtained by attaching a single $8$-cell to a buoquet of $4$-spheres. We denote 
\[ \PP \DD_3^8 = \mbox{ the collection of } 3\mbox{-connected } 8\mbox{-dimensional Poincar\'{e} duality complexes. }\]
The notation $M_k\in \PP\DD_3^8$ assumes that $\Rank(H_4(M_k))=k$. The attaching map of the $8$-cell is denoted by  $L(M_k)$, and is of the form (once we have chosen a basis $\{\alpha_1,\dots,\alpha_k\}$ of $\pi_4(M_k)\cong \Z^k$)
\begin{myeq}\label{Eq_general L(M) define}
        L(M_k)= \sum_{1\leq i< j\leq k}g_{i,j}[\alpha_i,\alpha_j] + \sum_{i=1}^{k}g_{i,i}\nu_i + \sum_{i=1}^{k}l_i\nu'_i. 
\end{myeq}
The matrix $\big((g_{i,j})\big)$ is the matrix of the intersection form, and hence, is invertible.  The notation $\nu_i$ stands for $\alpha_i\circ \nu$ and $\nu'_i$ for $\alpha_i \circ \nu'$. Here  $\nu$ is the Hopf map, and $\nu'\in \pi_7(S^4)$ is the generator for the $\Z/(12)$ factor satisfying $[\iota_4,\iota_4] = 2\nu + \nu'$. For such complexes, we consider 
\[ \PP(M_k)= \mbox{ the set of principal } SU(2) \mbox{-bundles } E(\psi)\stackrel{\psi}{\to} M_k \mbox{ such that } E(\psi) \mbox{ is }3\mbox{-connected}.\]   
The bundle $\psi$ is classified by a primitive element $\psi \in H^4(M_k)$, which satisfies a criterion (see Proposition \ref{Prop_tau map}). In this context, we first encounter the question whether $\PP(M_k)$ is non-empty. We prove (see Proposition \ref{existence of classifying map} and Proposition \ref{Prop_existence of classifying map odd case})
\begin{thma}
For $k\geq 3$, the set $\PP(M_k)$ is non-empty. 
\end{thma}
For $k=2$, there are examples where $\PP(M_k)$ is empty. This means that for every principal $SU(2)$-bundle over such complexes, the total space has non-trivial $\pi_3$. The idea here is that the existence of $\psi$ is given by a certain equation in $k$ variables, and solutions exist once $k$ is large enough. 

In the case of simply connected $4$-manifolds, the first kind of classification of circle bundles is the result of Giblin\cite{Gib68} which states 
\[ \mbox{ If } X=S^2\times S^2, \mbox{ then } X(\alpha) \simeq S^2 \times S^3 \mbox{ for any primitive } \alpha.\]
We also have an analogous result in the $8$-dimensional case 
\[ \mbox{ If } \psi \in \PP(S^4\times S^4), \mbox{ then } E(\psi) \simeq S^4 \times S^7 .\]
In fact, this fits into a more general framework. We call a manifold $M_k\in \PP\DD_3^8$ {\it stably trivial} if $L(M_k)$ is stably null-homotopic (that is the stable homotopy class of $L(M_k) : S^7 \to \big(S^4\big)^{\vee k}$ is $0$). In terms of \eqref{Eq_general L(M) define}, this means for every $i$, $g_{i,i}-2 l_i \equiv 0 \pmod{24}$. We have the following theorem (see Proposition \ref{stabtriv})
\begin{thmb}
Suppose $M_k$ is stably trivial. Then, for every $\psi \in \PP(M_k)$, $E(\psi)\simeq \#^{k-1} S^4\times S^7$, a connected sum of $k-1$ copies of $S^4\times S^7$. 
\end{thmb}
This directly generalizes the result for circle bundles over simply connected $4$-manifolds that are spin (identifying the spin manifolds as those whose attaching map is stably null). 

We proceed towards a more general classification of the homotopy type of the space $E(\psi)$ for $\psi \in \PP(M_k)$. 
Let 
$\PP\DD^{11}_{4,7}$ be the class of $3$-connected $11$-dimensional Poincar{\'e} duality complexes $E$ such that $E\setminus \{pt\}\simeq $ a wedge of $S^4$ and $S^7$. We first observe that $E(\psi)\in \PP \DD_{4,7}^{11}$ (see Proposition \ref{Prop_phi i is 0}), and we try to address the question of the classification of complexes in  $\PP \DD_{4,7}^{11}$ up to homotopy equivalence. The homology of such complexes $E$ is given by 
\[ H_m(E) \cong \begin{cases} \Z & m=0, 11 \\ 
                                          \Z^r & m=4,7 \\ 
                                            0 & \mbox{otherwise}.\end{cases} \]
We denote the number $r$ by $\Rank(E)$. The classification works differently for $r=1$, and for $r\geq 2$. Table \ref{table:homotopy equivqlence of E} lists the various possibilities for $r=1$. For $r\geq 2$, $E$ is a connected sum of copies of  $S^4 \times S^7$, and the complexes $E_{\lambda,\epsilon,\delta}$ defined below. Note that 
\begin{align*} 
\pi_{10}(S^{4}\vee S^{7})&\cong \pi_{10}(S^{4})\oplus \pi_{10}(S^{7})\oplus \pi_{10}(S^{10}) \\
 &\cong \Z/(24)\{x\}\oplus \Z/(3)\{y\} \oplus \Z/(24)\{\nu_7\}\oplus\Z\{[\iota_4,\iota_7]\}.
\end{align*}
Here, $x= \nu\circ \nu_7$ and $y= \nu' \circ \nu_7$. Let
\[    \phi_{\lambda, \epsilon, \delta}= [\iota_{4},\iota_{7}]+ \lambda (\iota_{7}\circ \nu_{7}) + \epsilon (\iota_{4}\circ x )  + \delta (\iota_{4}\circ y),
\] 
and define,  
$$    E_{\lambda, \epsilon, \delta}= (S^{4}\vee S^{7})\cup_{\phi_{\lambda, \epsilon, \delta}}D^{11}.
$$
The attaching map of the top cell of $E$ takes the form 
\[ L(E) : S^{10} \to \big(S^4 \vee S^7\big)^{\vee r}.\]
The stable homotopy class of $L(E)$ lies in 
\[\pi_{10}^s \Big(\big(S^4 \vee S^7\big)^{\vee r}\Big) \cong \big(\Z/(24)\{\nu\} \oplus \Z/(2)\{\nu^2\}\big)^{\oplus r}.\]
This takes the form $\lambda_s \beta\circ \nu + \epsilon_s \alpha \circ \nu^2$ for some $\beta\in \pi_7\Big(\big(S^4 \vee S^7\big)^{\vee r}\Big)$ and $\alpha \in \pi_4\Big(\big(S^4 \vee S^7\big)^{\vee r}\Big)$. Up to a change of basis we may assume that $\lambda_s \mid 24$, and if $\lambda_s$ is even, $\epsilon_s\in \Z/(2)$. These numbers are invariant over the homotopy equivalence class of $E$, and are denoted by $\lambda_s(E)$, and $\epsilon_s(E)$ (defined only if $\lambda_s(E)$ is even). We use these invariants to classify the homotopy types of elements in $\PP\DD_{4,7}^{11}$ (see Theorem \ref{Th_general classification of E})
\begin{thmc}
    Let $E \in \PP\DD^{11}_{4,7}$. Then the homotopy type of $E$ is determined by the following.
    \begin{enumerate}
        \item If $\lambda_s(E)$ is even and $\epsilon_s(E)=0$, then 
$$E \simeq \#^{r-1} E_{0,0,0} \# E_{\lambda_s, \epsilon, \delta} \quad \text{where } \epsilon \equiv 0 \pmod{2}.$$
        \item If $\lambda_s(E)$ is even and $\epsilon_s(E)=1$, then 
$$\begin{aligned}
            &E\simeq
            \#^{r-1} E_{0,0,0} \# E_{\lambda_s, \epsilon, \delta} \quad & &\text{where } \epsilon \equiv 1 \pmod{2}\\
          \text{or }  &E\simeq \#^{r-2} E_{0,0,0} \#E_{0,1,0} \# E_{\lambda_s, \epsilon, \delta} \quad & &\text{where } \epsilon \equiv 0 \pmod{2}.
        \end{aligned}
        $$
        \item If $\lambda_s(E)$ is odd, then 
$$E \simeq \#^{r-1} E_{0,0,0} \# E_{\lambda_s, \epsilon, \delta} \quad \text{or} \quad E \simeq \#^{r-2} E_{0,0,0} \#E_{0,1,0} \# E_{\lambda_s, \epsilon, \delta}.$$
    \end{enumerate}
    Further given $\lambda_s$, the choices of $\epsilon$ and $\delta$ are those which are mentioned in Table \ref{table:homotopy equivqlence of E}.
\end{thmc}

We see that in the list given in Table  \ref{table:homotopy equivqlence of E}, for certain cases the homotopy type of $E$ is  determined by $\lambda_s(E)$ and $\epsilon_s(E)$. This happens if $\lambda_s(E)=$ $0$, or $12$. We also observe that the homotopy type of $\Omega E$ depends only on the rank $r$. Now, we look at $M_k\in \PP\DD_3^8$, and try to determine the set of homotopy equivalence classes of $E(\psi)$ for $\psi\in \PP(M_k)$. In this process, we determine a formula for $\lambda(\psi) := \lambda_s(E(\psi))$ (Proposition \ref{Prop_tau j calculate}), and using this we determine the set of possible values of $\lambda_s(\psi)$ for $\psi\in \PP(M_k)$. 
The stable homotopy class of $L(M_k)$ lies in 
\[\pi_7^s \Big((S^4 )^{\vee k}\Big) \cong \big(\Z/(24)\{\nu\} \big)^{\oplus k}.\]
This takes the form $\sigma_s \alpha\circ \nu$ for some $\alpha\in \pi_4\big((S^4 )^{\vee k}\big)$, and up to a change of basis for $k\geq 2$, $\sigma(M_k):=\gcd(\sigma_s,24)$ is an invariant of the stable homotopy type of $M_k$. Other than $k$ and $\sigma(M_k)$, the explicit stable homotopy class of $\alpha$ above yields a linear map $\tau : H^4(M_k)\to \Z/(24)$ given by $\tau(\psi)= \psi(\sigma_s\alpha)$. We use the invariants $k$, $\sigma(M_k)$, $\tau$, and the intersection form to completely determine the possibilities of $\lambda(\psi)$ for $\psi\in \PP(M_k)$. (see Theorem \ref{Th_divisibility}, Proposition \ref{Prop_sigma k achieved}, Theorem \ref{Theorem_divisibility theorem for 3}, Theorem \ref{divfor4}, and Theorem \ref{oddcasethm})
\begin{thmd}
For any $\psi\in \PP(M_k)$, $\lambda(\psi)$ is a multiple of $\sigma(M_k)$ ($\pmod{24}$). Conversely, the multiples of $\sigma(M_k)$ that may be achieved are described as follows 
\begin{enumerate}
\item If the intersection form of $M_k$ is odd  and $k\geq 7$, then $\{\lambda(\psi) \mid \psi \in \PP(M_k)\}$ equals the set of multiples of $\sigma(M_k) \pmod{24}$. 
\item If the intersection form of $M_k$ is even, each $\psi \in \PP(M_k)$ satisfies $\epsilon_s(\psi)\equiv 0 \pmod{2}$. 
\item If $k\geq 7$, there are $\psi \in \PP(M_k)$ such that $\lambda(\psi)=\sigma(M_k)$, and also there are $\psi \in \PP(M_k)$ such that $\lambda(\psi)=3 \sigma(M_k)$. 
\item If $\sigma(M_k)\equiv 2$, or $4 \pmod{8}$ for $k\geq 5$, there is a $\psi\in \PP(M_k)$ such that $\lambda(\psi)\equiv 0 \pmod{8}$ if and only if the complex satisfies hypothesis $(H_8)$.  
\item If $\sigma(M_k)\equiv 2 \pmod{8}$ for $k\geq 5$, there is a $\psi\in \PP(M_k)$ such that $\lambda(\psi)\equiv 4 \pmod{8}$ if and only if the complex satisfies hypothesis $(H_4)$. 
\end{enumerate}
\end{thmd}
For lower values of $k$, we do not get systematic results like the above. That is, the set $\{\lambda(\psi) \mid \psi \in \PP(M_k)\}$ is not completely determined by $\sigma(M_k)$, $k$, $\tau$, and the intersection form. Theorem D implies that there are certain $M_k$ whose intersection form is even and there is no $\psi \in \PP(M_k)$ such that $E(\psi)\simeq \#^{k-1} S^4\times S^7$, however if the intersection form is odd, then for $k\geq 7$, there is a principal bundle $SU(2) \to \#^{k-1}(S^4\times S^7) \to M_k$. 

\begin{mysubsection}{Organization}
In \S \ref{htpyclass}, we prove a classification result for certain $3$-connected $11$-dimensional complexes which proves Theorem C.  In  \S \ref{sthtpyform}, we prove formulas relating the stable homotopy type of the total space to that of the base using the Chern character. The results for manifolds with even intersection form are proved in \S \ref{evencase}, and those with odd intersection form in  \S \ref{oddcase}. 
\end{mysubsection}







\section{Homotopy classification of certain $3$-connected $11$-complexes}\label{htpyclass}
We study $3$-connected $11$-dimensional Poincar{\'e} duality complexes $E$ such that $E \setminus \{pt\}$ is homotopic to a wedge of copies of $S^4$ and $S^7$. We write $\PP\DD^{11}_{4,7}$ for the collection of such complexes. 
Our target in this section is to analyze them up to homotopy equivalence. We show that these are classified by numbers $\lambda$, $\epsilon$ and $\delta$ which are explained in detail below.

\begin{mysubsection}{The rank one case} 
Let $E \in \PP\DD^{11}_{4,7}$ be such that $E \setminus \{pt\} \simeq S^4 \vee S^7$, that is, $\Rank(H_4(E))=1$. The homotopy type of $E$ is determined by the attaching map of the top cell, which is an element of 
\[\pi_{10}(S^{4}\vee S^{7})\cong \pi_{10}(S^{4})\oplus \pi_{10}(S^{7})\oplus \pi_{10}(S^{10}).\]
This must be of the form
\begin{myeq}\label{Eq_define phi lambda epsilon delta}
    \phi_{\lambda, \epsilon, \delta}= [\iota_{4},\iota_{7}]+ \lambda (\iota_{7}\circ \nu_{7}) + \epsilon (\iota_{4}\circ x )  + \delta (\iota_{4}\circ y).
\end{myeq} 
The total space associated with $\phi_{\lambda, \epsilon, \delta}$ is denoted by 
$$    E_{\lambda, \epsilon, \delta}= (S^{4}\vee S^{7})\cup_{\phi_{\lambda, \epsilon, \delta}}D^{11}.
$$

Note that as $(-\iota_4)\circ \nu= \nu + \nu'$, we have $(-\iota_4)\circ x= x+ y.$
For given any $\lambda,\epsilon$ and $\delta$; we observe the effect of the self homotopy equivalences on $E_{\lambda, \epsilon, \delta}$ as follows
  \begin{myeq}  \label{Eq_homotopy equivalences for rank one case}
  \begin{aligned}
        & \iota_4\mapsto -\iota_4,\quad \iota_7\mapsto -\iota_7  & &\implies E_{\lambda,\epsilon, \delta}\xrightarrow[]{\simeq}  E_{-\lambda, \epsilon,\epsilon -\delta}, \\
       &\iota_{4} \mapsto \iota_{4}, \quad \iota_{7}\mapsto \iota_{7} + a\iota_{4}\circ \nu & &\implies E_{\lambda,\epsilon, \delta}\xrightarrow[]{\simeq}  E_{\lambda, \epsilon + (\lambda +2)a,\delta}, \nonumber \\
       &\iota_{4} \mapsto \iota_{4}, \quad   \iota_{7}\mapsto \iota_{7} + b\iota_{4}\circ \nu' & &\implies E_{\lambda,\epsilon, \delta}\xrightarrow[]{\simeq}  E_{\lambda, \epsilon-4b ,\delta + (\lambda+ 1)b}. \nonumber 
    \end{aligned}
    \end{myeq}
This leads us to homotopy equivalences between $E_{\lambda, \epsilon, \delta}$'s depending on the choice of $\lambda \in \pi_{10}(S^7) \cong \Z/{24}$. Table \ref{table:homotopy equivqlence of E} lists the different homotopy types in $\PP \DD^{11}_{4,7}$ of rank $1$.

\begin{table}[htb]
\centering
\begin{tabular}{|c|c|c|}
\hline
$\lambda$  & $\#E_{\lambda, \epsilon, \delta}$'s & $E_{\lambda, \epsilon, \delta}$'s \\ \hline 
\hline
$0$ & $2$ & $E_{0,0,0}$, \quad $E_{0,1,0}$\\
\hline
$1$ & $3$ & $E_{1,0,0}$, \quad $E_{1,1,0}$, \quad $E_{1,2,0}$\\ 
\hline
 \multirow{1}{.5em}{$2$} & {$12$} &  $E_{2,0,0}$, \quad $E_{2,1,0}$, \quad $E_{2,2,0}$, \quad $E_{2,3,0}$, \quad $E_{2,0,1}$, \quad $E_{2,1,1}$, \\& & $E_{2,2,1}$, \quad $E_{2,3,1}$, \quad  $E_{2,0,2}$, \quad $E_{2,1,2}$, \quad $E_{2,2,2}$, \quad $E_{2,3,2}$\\ 
\hline
$3$ & $1$ & $E_{3,0,0}$\\
\hline
$4$ & $6$ & $E_{4,0,0}$, \quad $E_{4,1,0}$, \quad $E_{4,2,0}$, \quad $E_{4,3,0}$, \quad $E_{4,4,0}$, \quad $E_{4,5,0}$\\ 
\hline
$5$ & $3$ & $E_{5,0,0}$, \quad $E_{5,0,1}$, \quad $E_{5,0,2}$ \\
\hline
$6$ & $4$ & $E_{6,0,0}$, \quad $E_{6,1,0}$, \quad $E_{6,2,0}$, \quad $E_{6,3,0}$\\ 
\hline
$7$ & $3$ & $E_{7,0,0}$, \quad $E_{7,1,0}$, \quad $E_{7,2,0}$\\
\hline
$8$ & $6$ & $E_{8,0,0}$, \quad $E_{8,1,0}$, \quad $E_{8,0,1}$, \quad $E_{8,1,1}$, \quad $E_{8,0,2}$, \quad $E_{8,1,2}$\\
\hline
$9$ & $1$ & $E_{9,0,0}$\\
\hline
\multirow{1}{1em}{$10$} & $12$ & $E_{10,0,0}$, \quad $E_{10,1,0}$, \quad $E_{10,2,0}$, \quad $E_{10,3,0}$, \quad $E_{10,4,0}$, \quad $E_{10,5,0}$, \\& & $E_{10,6,0}$, \quad $E_{10,7,0}$, \quad $E_{10,8,0}$, \quad $E_{10,9,0}$, \quad $E_{10,10,0}$, \quad $E_{10,11,0}$\\ 
\hline
$11$ & $3$ & $E_{11,0,0}$, \quad $E_{11,0,1}$, \quad $E_{11,0,2}$\\
\hline
$12$ & $2$ & $E_{12,0,0}$, \quad $E_{12,1,0}$\\
\hline
\end{tabular}
\caption{Homotopy equivalence classes of $E_{\lambda, \epsilon, \delta}.$}
\label{table:homotopy equivqlence of E}
\end{table}

\end{mysubsection}

\begin{mysubsection}{A simplification of the attaching map}
We simplify and reduce the attaching map of the top cell of $E$.

\begin{prop}\label{Prop_attaching map of E} Let $E \in \PP\DD_{4,7}^{11}$ with $\Rank(E)=k-1$.
The attaching map $\phi$ of the top cell of $E$ as in \eqref{Eq_define phi lambda epsilon delta} can be reduced, up to homotopy, to the following form 
$$\phi= \sum_{i=1}^{k-1}[\iota_4^i,\iota_7^i]+ \sum_{i=1}^{k-1}\lambda_i \iota_7^i\circ \nu_{(7)}+ \sum_{i=1}^{k-1}s_i\nu_i\circ \nu_{(7)}+ \sum_{i=1}^{k-1}r_i\nu_i'\circ \nu_{(7)}.$$
\end{prop}

\begin{proof}
    By Hilton-Milnor decomposition, we have the following equivalence 
    \begin{multline*}
         \pi_{10}((S^4\vee S^7)^{k-1})\cong  \pi_{10}(S^4) ^{\oplus (k-1)} \oplus \pi_{10}(S^7)^{\oplus (k-1)} \oplus \pi_{10}(S^7) ^{\oplus {k-1\choose 2}} \oplus \\
         \pi_{10}(S^{10})^{\oplus {(k-1)\times (k-1)}} \oplus \pi_{10}(S^{10}) ^{\oplus {k-1 \choose 3}}.
    \end{multline*}
We choose $\eta_1,\dots, \eta_{k-1}\in \pi_4(E)$ and $\gamma_1,\dots, \gamma_{k-1}\in \pi_7(E)$ such that they correspond to the homology generators, say $\tilde{\eta}_1,\dots, \tilde{\eta}_{k-1}\in H_4(E)$ and  $\tilde{\gamma}_1,\dots, \tilde{\gamma}_{k-1}\in H_7(E)$ such that 
\begin{myeq}\label{Eq_cup product formula}
\tilde{\eta}_i^* \cup \tilde{\gamma}_j^*= \begin{cases}
    1 \quad \text{if } i=j,\\
    0 \quad \text{if } i \neq j.
\end{cases} \end{myeq}
Let $\tilde{f} \colon (S^4\vee S^7)^{\vee k-1} \to E$ be the inclusion which sends $\iota_4^i \mapsto \eta_i$ and $\iota_7^i \mapsto \gamma_i$ for $1 \leq i \leq k-1$.
Then $\tilde{f}\circ \phi\in\pi_{10}(E) $ whose image under the map $\rho\colon \pi_{10}(E)\rightarrow \pi_9(\Omega E)\rightarrow H_9(\Omega E)$ is $0$ in the tensor algebra $T(\tilde{\eta}_1,\dots, \tilde{\eta}_{k-1}, \tilde{\gamma}_1,\dots, \tilde{\gamma}_{k-1})/(\sum_{i=1}^{k-1}[\tilde{\eta}_i,\tilde{\gamma}_i])$.

The attaching map $\phi$ may contain triple Whitehead products as $[\iota^4_i,[\iota^4_j,\iota^4_{\ell}]]$, the Whitehead products of the form $[\iota_i^4,\iota_j^7]$ for $i\neq j$, the terms involving $[\iota_i^4,\iota_j^4]\circ \nu_7$ and $[\iota_i^4,\iota_j^4]\circ \nu_7'$. 
The triple Whitehead product maps injectively to the loop homology and hence they can not occur in the attaching map. The cup product formula in \eqref{Eq_cup product formula} implies that there is no Whitehead product of the form $[\iota_i^4,\iota_j^7]$ for $i\neq j$.
If $[\iota_i^4,\iota_j^4]\circ \nu_7$ and $[\iota_i^4,\iota_j^4]\circ \nu_7'$ appear in the attaching map, we update the map $\tilde{f}$ by appropriately sending $\iota_i^7 \mapsto \gamma_i-\eta_j\circ\nu'$, $\iota_i^7 \mapsto \gamma_i-\eta_j\circ\nu$ and $\iota_j^7 \mapsto \gamma_j- [\eta_i,\eta_j]$ to get the desired form of the attaching map.  
\end{proof}

We note that the composition is given by $$S^{10} \to (S^4 \vee S^7)^{\vee k-1} \to (S^7)^{\vee k-1}$$ which is an element of $(\pi_{10} S^7)^{\oplus k-1} \cong (\Z/{24} \{\nu\})^{\oplus k-1}$. This can be calculated using the real $e$-invariant, see \cite{Ada66}. We use this to reduce Proposition \ref{Prop_attaching map of E} to the case $\lambda_i=0$ for $i \leq k-2$.

\begin{prop}\label{Prop_attaching map phi}
Let $E \in \PP\DD^{11}_{4,7}$ and $\Rank(E)=k-1$. Then the attaching map $\phi$  of the top cell of $E$ can be reduced to the following form 
\begin{myeq}\label{Eq_reduced phi 1}
\phi= \sum_{i=1}^{k-1}[\iota_4^i,\iota_7^i]+ \lambda \iota_7^{k-1}\circ \nu_{(7)}+ \sum_{i=1}^{k-1}\epsilon_i\nu_i\circ \nu_{(7)}+ \sum_{i=1}^{k-1}\delta_i\nu_i'\circ \nu_{(7)}.
\end{myeq}
As a consequence $E \simeq  \#_{i=1}^{k-2} E_{0, \epsilon_i, \delta_i} \#  E_{\lambda, \epsilon_{k-1} , \delta_{k-1}}$.
\end{prop}

\begin{proof}
Let $\tau \colon H^{7}(E)\cong\mathbb{Z}\{\tilde{\gamma}_1^*,\tilde{\gamma}_2^*,\dots,\tilde{\gamma}_{k-1}^*\} \rightarrow \mathbb{Z}/24 $ be the linear map defined by 
$$\tau(\tilde{\gamma}_i^*)= e(r_i\circ \phi),  \quad \quad \text{for } 1\leq i\leq k-1,$$ where $e$ denotes for the real $e$-invariant and $r_i:(S^4 \vee S^7)^{\vee k-1} \to S^7$ is the retraction onto the $i$-th factor. 
Let $\tilde{\tau}\colon H^{7}(E)\rightarrow \mathbb{Z}$ be the lift of $\tau$ and $\lambda= \gcd (\tilde{\tau}(\tilde{\gamma}_1^*),\dots,\tilde{\tau}(\tilde{\gamma}_{k-1}^*)$.
Then we can change the generators $\tilde{\gamma}_1,\dots, \tilde{\gamma}_{k-1}$ such that $\tilde{\tau}(\tilde{\gamma}_i)=0$ for $1 \leq i< k-1$  and  $\tilde{\tau}(\tilde{\gamma}_{k-1})= \lambda.$
So, for a suitable choice of dual bases we can have $\phi$ as in \eqref{Eq_reduced phi 1}.
\end{proof}

\end{mysubsection}

\begin{mysubsection}{A general classification up to homotopy}
We now proceed to the classification of elements in $\PP\DD^{11}_{4,7}$. In the connected sum $E_{\lambda_1, \epsilon_1, \delta_1} \# E_{\lambda_2, \epsilon_2, \delta_2}$, we transform $\alpha_1= \alpha_1' + \alpha_2'$ and $\alpha_2= \alpha_2'$ by $A=\begin{pmatrix}
    1 & 1\\ 0& 1
\end{pmatrix}$. If $2|\lambda_2$, then we transform 
\[\beta_1=  \beta_1'-\epsilon_1\nu_2- \frac{\lambda_2\epsilon_1}{2}\nu_2' , \quad \beta_2 = -\beta_1'+\beta_2'-\epsilon_1[\alpha_1',\alpha_2'].\]
  Hence by the following expression
\begin{align*}
        &[\alpha_1,\beta_1]+[\alpha_2,\beta_2]+ \lambda_1\beta_1\circ\nu_{(7)}+ \lambda_2\beta_2\circ \nu_{7}+ \epsilon_1 x_1 + \epsilon_2 x_2+\delta_1 y_1+ \delta_2 y_2 \\
        =&[\alpha_1',\beta_1']+ [\alpha_2',\beta_2'] +(\lambda_1-\lambda_2)\beta_1'\circ\nu_{(7)}+ \lambda_2\beta_2'\circ\nu_{7} +\epsilon_1x_1 +(-\epsilon_1+\epsilon_2 + 2\lambda_2\epsilon_1-\lambda_1\epsilon_1)x_2 \\ 
        &\hspace{7cm} \delta_1 y_1+ (\delta_1+\delta_2+\lambda_2\epsilon_1(1+\lambda_1))y_2,
\end{align*}
we conclude
\begin{myeq}\label{Eq_gen connected sum 2 divides lambda2}
    E_{\lambda_1,\epsilon_1,\delta_1} \# E_{\lambda_2,\epsilon_2,\delta_2} \simeq E_{\lambda_1-\lambda_2, \epsilon_1, \delta_1}\# E_{\lambda_2,\epsilon_2 + \epsilon_1(2 \lambda_2 - \lambda_1-1), \delta_1+\delta_2+ (1  + \lambda_1)\epsilon_1 \lambda_2} \quad \quad \text{when } 2|\lambda_2.
\end{myeq}

\begin{prop}
For any unit $a \in \Z/24$, we have homotopy equivalence
    $$E_{\lambda, \epsilon, \delta} \# E_{0,0,0}  \simeq \begin{cases} E_{a\lambda, \epsilon, \delta} \# E_{0,0,0} \quad &\text{if } a \equiv 1 \pmod{3}  \\ E_{-a\lambda, \epsilon, \delta} \# E_{0,0,0} \quad &\text{if } a \equiv 2 \pmod{3}.
    \end{cases}$$
\end{prop}

\begin{proof}
    We transform \begin{align*}
        &\alpha_1= a \alpha_1' + b \alpha_2' & & \beta_1= a\beta_1'-24\beta_2'-b\epsilon \nu_2- 24b\epsilon\nu_1\\
        &\alpha_2= 24 \alpha_1' + a \alpha_2' & & \beta_2= -b\beta_1'+a\beta_2'-\epsilon b[\alpha_1',\alpha_2'] \nonumber
    \end{align*} with $a^2-24b=1$ and calculate $[\alpha_1,\beta_1]+[\alpha_2,\beta_2]+ \lambda \beta \circ\nu_{(7)} + \epsilon x_1 + \delta y_2$. This gives the homotopy equivalence 
     \begin{myeq}\label{Eq_mult by unit 1}
         E_{\lambda,\epsilon,\delta}\# E_{0,0,0} \\
         \simeq 
          E_{a\lambda, a^2\epsilon, a \delta  +{{a}\choose 2}\epsilon} \# E_{0, -b^2\epsilon -b\lambda \epsilon, b \delta + {{b}\choose 2}\epsilon}.
     \end{myeq}
     This proves the proposition except for when $\lambda \equiv 0 \pmod{2}$ and $\epsilon, b \equiv 1 \pmod{2}$. In that case, we have $E_{0,1,0}$ instead of $E_{0,0,0}$ in the second component of the connected sum. From \eqref{Eq_gen connected sum 2 divides lambda2}, we get
     \begin{myeq}\label{Eq_connected sum with E010}
     E_{\lambda,\epsilon,\delta} \# E_{0,1,0} \simeq E_{\lambda, \epsilon, \delta}\# E_{0,1 + \epsilon(- \lambda-1), \delta}
     \end{myeq}
     which we apply following the equivalence in \eqref{Eq_mult by unit 1} for $\lambda$ even and $\epsilon$,$b$ odd. This concludes the proof.
\end{proof}

The transformations above further simplify the possibilities of $E \in \PP\DD^{11}_{4,7}$ listed in Proposition \ref{Prop_attaching map phi}.

\begin{prop}\label{Prop_E in last two component}
    Let $E \in \PP\DD^{11}_{4,7}$ and $\Rank(E)=k-1$. Then 
$$
        E\simeq \#^{k-3}E_{0,0,0}\#E_{0,\hat{\epsilon},0}\#E_{\lambda,\epsilon,\delta}$$
    for some $\lambda,\epsilon \in \mathbb{Z}/24$, $\delta\in \Z/3$, $\hat{\epsilon} \in \Z/2$.
\end{prop}

\begin{proof}
From Proposition \ref{Prop_attaching map phi}, we have the attaching map $\phi$ as in \eqref{Eq_reduced phi 1}. 
Repeated use of the homotopy equivalences in \eqref{Eq_homotopy equivalences for rank one case} and \eqref{Eq_connected sum with E010} gives the reduced form \eqref{Eq_reduced phi 1} as follows
    $$\phi= \sum_{i=1}^{k-1}[\iota_4^i,\iota_7^i]+ \lambda \iota_7^{k-1}\circ \nu_{(7)}+ \hat{\epsilon}x_{k-2}+ \epsilon x_{k-1} + \delta y_{k-1},$$
where $\lambda,\epsilon \in \mathbb{Z}/24$, $\delta\in \Z/3$, $\hat{\epsilon} \in \Z/2$.
\end{proof}

\begin{cor}\label{Cor_stable lambda and epsilon}
Let $E \in \PP\DD^{11}_{4,7}$ and $\Rank(E)=k-1$. Then stably we have 
\begin{myeq}\label{Eq_Stable attaching map E}
        \Sigma^{\infty}E\simeq \Sigma^{\infty}(S^4\vee S^7)^{\vee {k-2}}\vee \Sigma^{\infty}Cone(\lambda_s(E) \nu_{(11)} +\epsilon_s(E) x ),
\end{myeq}
where $x$ is the generator of the stable homotopy group $\pi_{14}(S^8)\cong \mathbb{Z}/2$ and $\lambda_s(E)\in \mathbb{Z}/24$, $\epsilon_s(E)\in \mathbb{Z}/2.$
Moreover, for $k-1\geq 2$ if $\lambda_s(E) \equiv 1 \pmod{2}$ in \eqref{Eq_Stable attaching map E}, then $\epsilon_s(E)=0\in \mathbb{Z}/2$.
\end{cor}
\begin{proof}
    From \eqref{Eq_reduced phi 1} and Proposition \ref{Prop_E in last two component},
    $$
        \Sigma^4\phi= \lambda\iota_{11}^{k-1}\circ \nu_{(11)}+ \hat{\epsilon}\iota_{8}^{k-2}\circ \nu_{(11)}+\epsilon\iota_8^{k-1}\circ \nu_{(11)}.$$
    Note that $\nu_{(11)}\in \pi^{s}_{3}(S^0)\simeq \mathbb{Z}/24$ and $\Sigma^4 (\nu\circ \nu_{(7)})\in \pi^{s}_{6}(S^0)\simeq \mathbb{Z}/2$ are the generators.
    If  $\hat{\epsilon}=0\in \mathbb{Z}/2$, the result readily follows. Otherwise,   if $\epsilon=1,\hat{\epsilon}=1\in \mathbb{Z}/2$,  we apply the transformation $\iota_8^{k-2}+\iota^{k-1}_{8}\mapsto \iota^{k-1}_8.$ If $\epsilon=0,\hat{\epsilon}=1\in \mathbb{Z}/2$, we interchange  $\iota_8^{k-2}$ and $\iota_8^{k-1}$ to deduce the result. 
\end{proof}

We see that the stable homotopy type of $E \in \PP\DD^{11}_{4,7}$ is determined by $\lambda_s(E)$ which is a divisor of $24$ and $\epsilon_s(E) \in \Z/2$.
The following theorem classifies the different homotopy types of $E$ given the values of $\lambda_s$ and $\epsilon_s$.

\begin{theorem}\label{Th_general classification of E}
    Let $E \in \PP\DD^{11}_{4,7}$ and $\Rank(E)= k-1$. Then depending on $\lambda_s=\lambda_s(E)$ and $\epsilon_s=\epsilon_s(E)$, the homotopy type of $E$ is determined by the following.
    \begin{enumerate}
        \item If $\lambda_s$ is even and $\epsilon_s=0$, then $$E \simeq \#^{k-2} E_{0,0,0} \# E_{\lambda_s, \epsilon, \delta} \quad \text{where } \epsilon \equiv \epsilon_s \pmod{2}.$$

        \item If $\lambda_s$ is even and $\epsilon_s=1$, then $$\begin{aligned}
            &E\simeq
            \#^{k-2} E_{0,0,0} \# E_{\lambda_s, \epsilon, \delta} \quad & &\text{where } \epsilon \equiv 1 \pmod{2}\\
          \text{or }  &E\simeq \#^{k-3} E_{0,0,0} \#E_{0,1,0} \# E_{\lambda_s, \epsilon, \delta} \quad & &\text{where } \epsilon \equiv 0 \pmod{2}.
        \end{aligned}
        $$

        \item If $\lambda_s$ is odd, then $$E \simeq \#^{k-2} E_{0,0,0} \# E_{\lambda_s, \epsilon, \delta} \quad \text{or} \quad E \simeq \#^{k-3} E_{0,0,0} \#E_{0,1,0} \# E_{\lambda_s, \epsilon, \delta}.$$
    \end{enumerate}
    Further given $\lambda_s$, the choices of $\epsilon$ and $\delta$ are those which are mentioned in Table \ref{table:homotopy equivqlence of E}.
\end{theorem}

\begin{proof}
    We write $Y_1=\#^{k-2} E_{0,0,0} \# E_{\lambda, \epsilon, \delta}$ and $Y_2=\#^{k-2} E_{0,0,0} \# E_{\lambda, \epsilon', \delta'}$, and let $Y_1 \stackrel{f}{\to} Y_2$ be a homotopy equivalence with homotopy inverse $g$. We show that in this case the pair $(\epsilon,\delta)$ is related to $(\epsilon',\delta')$ by the transformations \eqref{Eq_homotopy equivalences for rank one case}. There exists a unique (up to homotopy) factorization $(S^7)^{\vee k-1} \to (S^7 \vee S^4)^{\vee k-1} \to Y_2$ through cellular approximation $\tilde{f}$ as in the following diagram.
    $$\xymatrix{(S^7)^{\vee k-1} \ar[r] \ar[dr]_{\tilde{f}} \ar[drr]^{f_7} & \#^{k-2} E_{0,0,0} \# E_{\lambda, \epsilon, \delta} \ar[r]^f  & \#^{k-2} E_{0,0,0} \# E_{\lambda, \epsilon', \delta'}\\ &  (S^7 \vee S^4)^{\vee k-1} \ar[r]_{\rho} \ar@{-->}@/_1pc/[ur] & (S^7)^{\vee k-1} \ar[u]}$$
    Now we consider the composition $f_7\colon =\rho \circ \tilde{f}$ where $\rho$ is the projection map.
    From the stable homotopy type, we get an isomorphism $$\xymatrix{ f_7 \colon \stackbeloweq{\pi_7((S^7)^{\vee k-1})}{\Z\{\beta_1, \dots, \beta_{k-1}\}}  \ar[r]^{\cong} & \stackbeloweq{\pi_7((S^7)^{\vee k-1})}{\Z\{\gamma_1, \dots, \gamma_{k-1}\}}}$$
    where $f_7(\beta_{k-1})\equiv\gamma_{k-1}\pmod{24}$. Hence, the corresponding matrix of $f_7$ is 
$$\equiv \begin{pmatrix}
        & & & 0\\ & A & & \vdots\\ & & & 0 \\ * & \dots & * & 1\\ 
    \end{pmatrix} \pmod{24}$$ 
    for some $A \in GL_{k-2} (\Z)$. From the inverse homotopy equivalence $g \colon Y_2 \to Y_1$, we can construct a corresponding block matrix of $g_7$ similar to that of $f_7$ for some $B \in GL_{k-2}(\Z)$ $\pmod{24}$. 
    Through suitable pre-composition of $f_7$ and post-composition of $g_7$ we may assume that $A=B=I$ where $I$ is the identity matrix in $GL_{k-2} (\Z)$. Thus both $f_7$ and $g_7$ are composition of shearing maps with $\beta_{k-1} \mapsto \gamma_{k-1}$, that is, they are composition of maps associated to 
$\beta_i \mapsto \gamma_i + c_i \gamma_{k-1}$ for some $c_i \in \Z$.

We now consider the map $f$ on the $7$-skeleton  
 $$\xymatrix{(S^4\vee S^7)^{\vee k-1} \ar[r]^{f^{(7)}} \ar[d]^{\vee_{i=1}^{k-1} (\alpha_i \vee \beta_i)}  & (S^4\vee S^7)^{\vee k-1} \ar[d]^{\vee_{i=1}^{k-1} (\eta_i \vee \gamma_i)}\\
Y_1 \ar[r]^f  &Y_2,}$$
which takes the form
    \begin{align*}
    &\beta_{i}\mapsto \gamma_i + c_i \gamma_{k-1} + \sum_{j=1}^{k-1} a_{i,j} \nu_j + \sum_{j=1}^{k-1} a_{i,j}'\nu_j'+ \sum_{\substack{j=1\\ \ell=1 \\ j \neq \ell }}^{\substack{j=k-1\\ \ell=k-1}} a_{i,j, \ell}[\eta_j,\eta_{\ell}], & & \alpha_i \mapsto \eta_i, \quad \text{for } 1 \leq i \leq k-2,\\
    &\beta_{k-1} \mapsto \gamma_{k-1} +\sum_{j=1}^{k-1} b_{j} \nu_j + \sum_{j=1}^{k-1} b_{j}'\nu_j'+ \sum_{\substack{j=1\\ \ell=1 \\ j \neq \ell }}^{\substack{j=k-1\\ \ell=k-1}} b_{j ,\ell}[\eta_j,\eta_{\ell}],  & & \alpha_{k-1} \mapsto \eta_{k-1}-\sum_{j=1}^{k-2} c_j \eta_j.
\end{align*}
As $f$ is a homotopy equivalence, we must have that the attaching map of the $11$-cell of $Y_1$ must be carried by $f^{(7)}$ to the attaching map of $Y_2$, that is, $f^{(7)} \circ L(Y_1) \simeq L(Y_2)$. We now look at the coefficients of $\eta_{k-1}\circ x$ and $\eta_{k-1}\circ y$ that arises in $f^{(7)}\circ L(Y_1)$ and note
that the only terms which contribute to these coefficients are 
\[f^{(7)}( [\alpha_{k-1},\beta_{k-1}]+\lambda_s \beta_{k-1}\circ \nu_{(7)}+  \epsilon \alpha_{k-1} \circ x+ \delta \alpha_{k-1} \circ y).\] 
We now deduce 
$$\epsilon'=(\lambda_s+2)b_{k-1}-4b'_{k-1} \quad \quad \text{and} \quad  \quad \delta'= (\lambda_s+1)b_{k-1}',$$
which verifies the result for complexes of the type $ \#^{k-2} E_{0,0,0} \# E_{\lambda, \epsilon, \delta}$.

For the remaining cases, we follow the same argument with $Y_1 = \#^{k-3} E_{0,0,0} \#E_{0,1,0} \# E_{\lambda_s, \epsilon, \delta}$ with $\epsilon$ even if $\lambda_s$ is, and $Y_2=\#^{k-3} E_{0,0,0} \#E_{0,\hat{\epsilon},0} \# E_{\lambda_s, \epsilon', \delta'}$, where $\hat{\epsilon}=0$ or $1$. 
Note that  the only terms which contribute to $\eta_{k-2}\circ x$ are 
\[f^{(7)}( [\alpha_{k-1},\beta_{k-1}]+[\alpha_{k-2},\beta_{k-2}]+\alpha_{k-2}\circ x + \lambda_s \beta_{k-1}\circ \nu_{(7)}+  \epsilon \alpha_{k-1} \circ x+ \delta \alpha_{k-1} \circ y).\] 
A direct computation implies
\begin{myeq}\label{Eq_coeff of x k-2}
\lambda_s b_{k-2}+\epsilon c_{k-2}^2+1\equiv  \hat{\epsilon}  \pmod{2}.
\end{myeq}

First let $\lambda_s$ be even and $\epsilon\equiv 0\pmod{2}$. This implies $\hat{\epsilon}=1$. 
Finally, let $\lambda_s$ is odd. We look at the coefficients of $[\eta_{k-1},[\eta_{k-2}, \eta_{k-1}]]$ and $[\eta_{k-2}, \eta_{k-1}] \circ \nu_{(7)}$ in $f^{(7)}\circ L(Y_1)$ $\pmod{2}$, which are $b_{k-2,k-1} +c_{k-2} b_{k-1} -a_{k-2, k-1}$ and $\lambda_s b_{k-2, k-1} + a_{k-2, k-1} + b_{k-2} -c_{k-2}b_{k-1}-\epsilon c_{k-2}$. Since both these coefficients are zero, we have $b_{k-2}\equiv\epsilon c_{k-2}\pmod{2}$. Using the relation \eqref{Eq_coeff of x k-2}, we observe that $\hat{\epsilon}=1$. The conditions on $\epsilon'$ and $\delta'$ are verified analogously as in the previous case. This completes the proof of the various implications in the theorem.
\end{proof}

\end{mysubsection}

\begin{mysubsection}{The loopspace homotopy type}
We study the loop space homotopy type of $E \in \PP\DD^{11}_{4,7}$ with $E^{(7)} \simeq (S^4 \vee S^7)^{\vee k-1}$ and show that the loop space homotopy is independent of the $\lambda, \epsilon$ and $\delta$ occurring in the attaching map $\phi$ of $E$.

\begin{theorem}
    The homotopy type of the loop space of $E$ is a weak product of loop spaces on spheres, and depends only on $k-1=\Rank (H_4(E))$. In particular, $$\Omega E \simeq \Omega (\#^{k-1} (S^4 \times S^7) ).$$
\end{theorem}

\begin{proof}
    This follows from the arguments in \cite{BaBa18} and \cite{BaBa19}. More explicitly, we first compute the homology of $\Omega E$ via cobar construction, see \cite{Ada56}. In this case, $H_*(E)$ is a coalgebra which is quasi-isomorphic to $C_*(E)$, and hence we may compute the cobar construction of $H_*(E)$ and deduce as in \cite[Proposition 2.2]{BaBa19}
    $$H_*(\Omega E) \cong T(a_1, b_1, \dots, a_{k-1}, b_{k-1})/ (\sum [a_i, b_i])$$ where $\rho(\alpha_i)=a_i$ and $\rho(\beta_i)=b_i$ with $\rho$ defined as $$\rho \colon \pi_r (E) \cong \pi_{r-1} (\Omega E) \xrightarrow{Hur} H_{r-1}(\Omega E).$$
    We then note that $H_*(\Omega E)$ is the universal enveloping algebra of the graded Lie algebra $L(a_1, b_1, \dots, a_{k-1}, b_{k-1})/ (\sum [a_i, b_i])$ where $L$ is the free Lie algebra functor. Now we apply the Poincar{\'e}-Birkhoff-Witt theorem as in \cite[Proposition 3.6]{BaBa19} to deduce the result. 
    \end{proof}

\end{mysubsection}

\section{Stable homotopy type of the total space.}\label{sthtpyform}
In this section, we examine the possible stable homotopy types of the total space $E$ for a principal $SU(2)$-fibration over $M_k \in \PP\DD^8_3$. We relate this to the stable homotopy type of $M_k$. 

Let $f \colon M_k \to \HH P^{\infty}$ be a map such that $\pi_4(f)\colon \pi_{4}(M_k)\rightarrow \pi_{4}(\mathbb{H}P^{\infty})\cong \mathbb{Z}$ is surjective.
This ensures that the homotopy fibre $E(f)$ is $3$-connected and is a Poincar\'{e} duality complex of dimension $11$. One easily deduces
\[ H_i(E(f))  = \begin{cases} 
    \mathbb{Z} &\quad\text{}  i= 0,11   \\
    \mathbb{Z}^{\oplus{k-1}}&\quad\text{}  i= 4,7\\
    0&\quad\text{otherwise}
\end{cases}      \]
from the Serre spectral sequence associated to the fibration $S^3 \to E(f) \to M_k$.
We may now consider a minimal CW-complex structure on $E:=E(f)$ with $(k-1)$ $4$-cells, $(k-1)$ $7$-cells, and one $11$-cell, see \cite[Section 2.2]{Hat01}.
The $7$-th skeleton $E^{(7)}$ is, therefore, a pushout
\begin{myeq}\label{Eq_pushout diagram E7}
\xymatrix{
( S^{6})^{\vee{(k-1)}}\ar[r] \ar[d]_{\vee_{i=1}^{k-1}\phi_i} & ({D^{7}})^{\vee{(k-1)}} \ar[d] \\
(S^{4})^{\vee{(k-1)}}  \ar[r] & E^{(7)}}
\end{myeq}
We now observe that the $\phi_i$ are all $0$.

 \begin{prop}\label{Prop_phi i is 0}
 The maps $\phi_i \simeq 0$ for $1 \leq i \leq k-1$.
 \end{prop}
\begin{proof}
    Note that the homotopy class of each of  the attaching maps  is in $\pi_6(S^4)^{\oplus {k-1}}$ which lies in the stable range.
Applying the $\Sigma^{\infty}$ functor on the diagram \eqref{Eq_pushout diagram E7}, we get the cofibre sequence 
$$\Sigma^{\infty}(S^6)^{\vee{(k-1)}}\xrightarrow{\Sigma^{\infty}(\vee_{i=1}^{k-1}\phi_i)}  \Sigma^{\infty}(S^4)^{\vee{(k-1)}}\rightarrow \Sigma^{\infty} E^{(7)}$$ which in turn induces a long exact sequence (on the stable homotopy groups)
\begin{myeq}\label{Eq_LES of stable homotopy of E7}
\dots\rightarrow \pi_{7}^{s}(E^{(7)}) \rightarrow \pi_6^{s}(S^6)^{\oplus k-1} \xrightarrow{\Phi} \pi_6^{s}(S^4)^{\oplus k-1} \rightarrow \pi_{6}^{s}(E^{(7)})\rightarrow \dots
\end{myeq}
where $\Phi$ is the induced map of $\vee_{i=1}^{k-1} \phi_i$. We have the following commutative diagram 
$$\xymatrix{
 & \pi_6^{s}(S^{4})^{\oplus k-1} \ar[r] \ar[d] & \pi^{s}_6(E)\ar[d] &\pi^{s}_6(E^{(7)})\ar[l]^{\simeq}  \\
\stackbelow{\pi^s_6(S^7)}{0} \ar[r] & \pi_6^{s}(S^{4})^{\oplus k} \ar[r] &\pi_{6}^{s}(M_k).}$$
The second row is a part of a long exact sequence and so, the map $\pi_6^s(S^4)^{\oplus k}\to \pi^s_6(M_k)$ is injective. Hence, the map $ \pi_6^{s}(S^4)^{\oplus{(k-1)}}\rightarrow \pi_6^{s}(E)$ is injective, in \eqref{Eq_LES of stable homotopy of E7} $\Phi$ is forced to be $0$. The result follows.
\end{proof}

Proposition \ref{Prop_phi i is 0} implies that $E$ fits into the pushout $$
\xymatrix{
 S^{10}\ar[r] \ar[d]_{L(E)} & {D^{11}} \ar[d] \\
(S^{4}\vee S^7)^{\vee{(k-1)}}  \ar[r] & E}
$$ for some $[L(E)]\in \pi_{10}((S^4\vee S^7)^{\vee(k-1)})$.
Hence, $E$ belongs to $\PP\DD^{11}_{4,7}$. We consider $$ \xymatrix{S^{10} \xrightarrow{L(E)} (S^4 \vee S^7)^{\vee{(k-1)}} \xrightarrow{} (S^7)^{\vee{(k-1)}}}$$ which is of the form $$
    \sum_{i=1}^{k-1}\lambda_i \iota^{i}_{7}\circ \nu_{7}\in \pi_{10}(S^7) ^{\oplus{k-1}} \cong \bigoplus_{i=1}^{k-1}\mathbb{Z}/24\{\nu_{7}\}.$$
The coefficients $\lambda_i$ can be computed via the $e$-invariant, see \cite{Ada66}. Recall that the $e$-invariant of a map $g \colon S^{11} \to S^8$ can be computed using Chern character. The complex $K$-theoretic $e$-invariant $e_{\CC}$ is computed via the diagram 
$$\xymatrix{
 0\ar[r]& \stackupeq{\tilde{K}(S^{12})}{\Z\{b_{12}\}} \ar[r] \ar[d]^{ch} & \tilde{K}(\Sigma C_{g})\ar[r]\ar[d]^{ch} & \stackupeq{\tilde{K}(S^{8})}{\Z\{b_8\}} \ar[r]\ar[d]^{ch}& 0  \\
   0\ar[r]& \stackbeloweq{\tilde{H}^{ev}(S^{12};\mathbb{Q})}{\Q\{a_{12}\}} \ar[r] &\tilde{H}^{ev}(\Sigma C_{g};\mathbb{Q})\ar[r]& \stackbeloweq{\tilde{H}^{ev}(S^{8};\mathbb{Q})}{\Q\{a_8\}}\ar[r]&0.    
} $$
We obtain $$ch(b_{12})= a_{12}, \quad ch(b_{8})= a_{8} + r{a}_{12}.$$
If $g=\lambda \nu_{(7)}$, $e_{\CC}(g)=r=\frac{\lambda}{12} \in \Q/\Z$, see \cite{Ada66}. We also have $e_{\CC}=2e$, where $e$ is computed using the Chern character of the complexification $c \colon KO \to K$. Therefore, from \cite[Proposition 7.14]{Ada66}
\begin{myeq}\label{Eq_Adam's result}
    b_8 \in Im(c) \implies e(g)=\frac{r}{2} \in \Q / \Z.
\end{myeq}

\begin{mysubsection}{$K$-theory of $M_k$}
Consider the Atiyah-Hirzebruch spectral sequence $$E_2^{**}= H^{*}(M_k; \pi_*K) \implies K^{*}(M_k).$$
As $M$ has only even dimensional cells, this has no non-trivial differential for degree reasons. This gives the additive structure of $K^{0}(M_k)$. Let $H^4(M_k) \cong \Z \{\psi_1,\dots, \psi_{k}\}$ and $H^8(M_k) \cong \Z\{z\}$. Note that if $\alpha_1,\dots, \alpha_k\in \pi_4(M_k) \cong H_4(M_k)\cong \Z^k$ is dual to the basis $\psi_1, \dots, \psi_k$, in the expression \eqref{Eq_general L(M) define}, the matrix $\big((g_{i,j})\big)$ of the intersection form is related to the cup product via the equation $\psi_i\cup \psi_j = g_{i,j} z$. 
Let $\tilde{\psi_1},\dots, \tilde{\psi_{k}}, \tilde{z}$ be classes in $K(M_k)$ corresponding to $\psi_1,\dots, \psi_{k}, z$ respectively, in the $E^{\infty}$-page. We have the diagram
$$\xymatrix{
 0\ar[r]&\tilde{K}^{0}(S^{8}) \ar[r]^{q^*} \ar[d]^{ch} & \tilde{K}^{0}(M_k)\ar[r]^{i^{*}}\ar[d]^{ch} &\tilde{K}^{0}((S^{4})^{\vee{k}})\ar[r]\ar[d]^{ch}& 0  \\
    0\ar[r]&\tilde{H}^{ev}(S^{8}; \mathbb{Q}) \ar[r] &\tilde{H}^{ev}(M_k;\mathbb{Q})\ar[r]& \tilde{H}^{ev}((S^{4})^{\vee{k}};\mathbb{Q})\ar[r]&0  }$$
where 
\begin{myeq}\label{Eq_ch in K of M}
ch(\tilde{z})= z,\quad ch (i^*(\tilde{\psi_j}))= {\psi_j} (\implies ch(\tilde{\psi_j})= {\psi_j}+ \tau_j z), 1\leq j\leq k.\end{myeq}
Note that in terms of the formula \eqref{Eq_general L(M) define}, $\psi_i \psi_j=g_{i,j}z$. We now use the fact that $ch$ is a ring map to get $$ch(\tilde{\psi_i}\tilde{\psi_j})= (\psi_i+\tau_i z)\cup (\psi_j+\tau_j z)= \psi_i\cup \psi_j= g_{i,j}z.$$
As $ch \colon K(M_k) \to H^{ev}(M_k;\Q)$ is injective, we deduce $\tilde{\psi_i} \tilde{\psi_j}=g_{i,j} \tilde{z}$.

Further let $q_i \colon (S^4)^{\vee k} \to S^4$ be the retraction onto the $i$-th factor, and note that $q_i \circ L(M)$ is stably equivalent to $(g_{i,i} -2l_i) \alpha_i \circ \nu_{(7)}$. Thus the $e$-invariant $$
e_{\CC}(q_i\circ L(M))=  (g_{i,i}-2l_i)e_{\CC}(\Sigma \nu_{(7)})= \frac{g_{i,i}-2l_i}{12}\in \mathbb{Q}/\mathbb{Z}.$$ We summarize these observations in the following proposition.

\begin{prop}\label{KMk}
Let $M_k \in \PP\DD^8_3$ with $L(M_k)$ as in \eqref{Eq_general L(M) define}. Then, ${K}^{0}(M_k) \cong \Z\{1, \tilde{\psi}_1,\tilde{\psi}_{2},\dots, \tilde{\psi_{k}},\tilde{z} \}$. The ring structure is given by $$
     \tilde{\psi}_i \tilde{\psi}_j= g_{i,j} \tilde{z} \quad \text{ for } 1\leq i, j\leq k. $$
and $$e_{\CC}(q_i\circ L(M))= \frac{g_{i,i}-2l_i}{12}\in \mathbb{Q}/\mathbb{Z} \quad \text{ for } 1\leq i \leq k.$$
\end{prop}

\end{mysubsection}

\begin{mysubsection}{$K$-theory of $E(f)$}
The space $E:=E(f)$ is the total space of the sphere bundle associated to the quaternionic line bundle classified by $f$. We note that the quaternionic line bundle has a complex structure, and therefore, has a $K$-orientation.

Let $\gamma_{\mathbb{H}}$ be the canonical $\mathbb{H}$-bundle over $\mathbb{H}P^{\infty}$. The $K$-theoretic Thom class of $\gamma_{\mathbb{H}}$ is given by $\gamma_{\mathbb{H}}-2\in \tilde{K}^{0}(\mathbb{H}P^{\infty})$. 
As the total space of the sphere bundle is contractible, the Thom space $Th(\gamma_{\HH}) \simeq \HH P^{\infty}$. Consider the map $\pi \colon \mathbb{C}P^{\infty} \to \mathbb{H}P^{\infty}$.
The pullback bundle  $\pi^*(\gamma_{\mathbb{H}})=\gamma_{\CC} \oplus \bar{\gamma}_{\CC}$, where $\gamma_{\CC}$ is the canonical line bundle over ${\CC}P^{\infty}$.
Therefore, 
$$\pi^*ch(\gamma_\mathbb{H}-2) =ch(\pi_{}^{*}(\gamma_\mathbb{H}-2)) = ch(\gamma_{\mathbb{C}})+ ch(\bar{\gamma}_{\mathbb{C}})-2=e^x + e^{-x} -2,$$
where $H^\ast(\C P^\infty;\Z) \cong \Z[x]$ and $x=c_1(\gamma_\C)$. Since $\pi^*$ is injective on $H^*$, we may use this formula to deduce 
\begin{myeq}\label{Eq_chern phi gamma H}
ch(\Phi_{K}(\gamma_{\mathbb{H}}))= \Phi_H(\gamma_{\mathbb{H}})(1+  \frac{y}{12}+ \frac{y2}{360}+\dots)
\end{myeq}
where $H^*(\HH P^{\infty}) \cong \Z[y]$ and $\pi^*(y)=x^2$. We use the Thom isomorphism associated to $f^*(\gamma_{\HH})$ to deduce the following.

\begin{prop}\label{Prop_K 0 of Th and M}
Assume that $f^*(y)=\psi_k$. Then,
 $$\tilde{K}^{0}(Th(f^{*}(\gamma_{\mathbb{H}})))\cong \tilde{K}^{0}(M)\{\Phi_{K}(f^{*}\gamma_{\mathbb{H}})\} $$ 
as a $\tilde{K}^{0}(M)$-module, and 
 $$ch(\Phi_{K}(f^{*}\gamma_{\mathbb{H}}))= \Phi_H(f^*\gamma_{\mathbb{H}}) (1 + \frac{\psi_k^{}}{12}+ \frac{\psi_k^{2}}{360}+ \dots).$$
\end{prop}

\begin{proof} 
From the naturality of the Chern character as well as the Thom class, we have $$
      ch(\Phi_{K}(f^{*}(\gamma_{\mathbb{H}}))= ch(f^{*}(\Phi_{K}((\gamma_{\mathbb{H}}))= f^{*}ch(\Phi_{K}(\gamma_{\mathbb{H}})).$$
The result follows from \eqref{Eq_chern phi gamma H} and the fact $f^*(y)=\psi_k$.
\end{proof}

\begin{notation}\label{Nota_matrix form}
    Suppose we are in the situation of Proposition \ref{Prop_K 0 of Th and M}, that is, ${f^*(y)=\psi_k}$. We now assume that the basis $\{ \psi_1, \dots, \psi_k \}$ is such that one of the following cases occur
\begin{itemize}
    \item[Case 1]  If $\psi_k\cup \psi_k= \pm z$, then $\psi_j \psi_k=0$ for $1\leq j \leq k-1$.
    \item[Case 2]  If $\psi_k\cup \psi_k= g_{k,k}z$ for some integer $g_{k,k}\neq \pm 1,$ then assume $\psi_{k-1} \psi_k=1$ and $\psi_j \psi_k=0$ for $1 \leq j \leq k-2$.      
\end{itemize}    
\end{notation}

In terms of these notations we prove the following calculation. Here we consider the cofibre sequence $$E\rightarrow M_k \rightarrow Th(f^*(\gamma_{\mathbb{H}}))\rightarrow \Sigma E\rightarrow \Sigma M\rightarrow \dots $$ which demonstrates $K(\Sigma E)$ as a submodule of $K(Th(f^* \gamma_{\HH}))$ because $K(\Sigma M_k)=0$. The following proposition identifies this submodule.

\begin{prop}\label{Prop_K of sigma E}  
(1) Suppose we are in Case 1 of Notation \ref{Nota_matrix form}, then we have 
 $$\tilde{K}^{0}(\Sigma E)\cong\mathbb{Z}\{ 
   \Phi_K(f^*(\gamma_\mathbb{H}))\tilde{\psi}_1,\dots,\Phi_K(f^*(\gamma_\mathbb{H}))\tilde{\psi}_{k-1} , 
    \Phi_K(f^*(\gamma_\mathbb{H}))\tilde{z}\}.$$
(2) Suppose we are in Case 2 of Notation \ref{Nota_matrix form}, then
$$\tilde{K}^{0}(\Sigma E)\cong\mathbb{Z}\{ 
   \Phi_K(f^*(\gamma_\mathbb{H}))\tilde{\psi}_1,\dots,\Phi_K(f^*(\gamma_\mathbb{H}))\tilde{\psi}_{k-2} , \Phi_K(f^*(\gamma_{\mathbb{H}}))(\tilde{\psi}_k-g_{k,k}\tilde{\psi}_{k-1}),
    \Phi_K(f^*(\gamma_\mathbb{H}))\tilde{z}\}.$$
\end{prop}

\begin{proof}
We have  the following short exact sequence
$$\xymatrix{
 0\ar[r]&\tilde{K}^{0}(\Sigma E) \ar[r]  & \tilde{K}^{0}(Th(f^*(\gamma_{\mathbb{H}})) \ar[r]^-{s_0}   &\tilde{K}^0(M)\ar[r] & 0   }$$
 which implies that 
$$        \tilde{K}^{0}(\Sigma E)= \mbox{Ker} (s_0\colon \tilde{K}^{0}(Th(f^*(\gamma_{\mathbb{H}}))\rightarrow \tilde{K}^0(M))$$ 
where $s_0$ is the restriction along the zero section. Note that $s_0$ is a $\tilde{K}^{0}(M)$-module map.
Hence, 
\begin{align*}
s_0(\Phi_{K}(f^*\gamma_{\mathbb{H}}))\tilde{z})&= e_{K}(f^*\gamma_{\mathbb{H}})\tilde{z}=0 \nonumber\\
    s_0(\Phi_{K}(f^*\gamma_{\mathbb{H}}))\tilde{\psi}_i)&= e_{K}(f^*\gamma_{\mathbb{H}})\tilde{\psi}_i= g_{i k}\tilde{z} , 1\leq i \leq k 
\end{align*}
since $e_{K}(f^*\gamma_{\mathbb{H}}))= \tilde{\psi}_k + m\tilde{z} $ for some $m$.
The result follows from a direct calculation of the kernel and the assumptions in the respective cases.
\end{proof}
\end{mysubsection}

We now choose the various maps $\chi_j \colon S^7 \to E$ for $1 \leq j \leq k-1$ such that on $K$-theory they precisely represent the choice of the first $(k-1)$-elements in the basis of Proposition \ref{Prop_K of sigma E}. In these terms we calculate the $e_{\CC}$-value of the composite 
$$r_j \circ L(E) \colon S^{10} \stackrel{L(E)}{\to} (S^4 \vee S^7)^{\vee k-1}\stackrel{r_j}{\to} S^7$$
 where $r_j$ is the restriction onto the $j$-th factor, and $L(E)$ is the attaching map of the top cell of $E$.

\begin{prop}\label{Prop_tau j calculate}
(1) If we are in Case 1, then $$
      e_{\CC}(r_j\circ L(E))=\tau_j=  \frac{g_{j,j}-2l_j}{12}\in \mathbb{Q}/\mathbb{Z}, 1\leq j\leq k-1 .$$
(2) If we are in Case 2, then $$
     e_{\CC}(r_j\circ L(E))= \tau_j= \frac{g_{j,j}-2l_j}{12}\in \mathbb{Q}/\mathbb{Z}, 1\leq j\leq k-2; \quad e_{\CC}(r_{k-1}\circ L(E) )= \tau_k- g_{k,k}\tau_{k-1} \in \mathbb{Q}/\mathbb{Z}.$$
\end{prop}

\begin{proof}
For the $e$-invariant, we calculate the Chern character $ch\colon \tilde{K}^{0}(Th(f^*\gamma_{\mathbb{H}}))\rightarrow H^{ev}(Th(f^*\gamma_{\mathbb{H}});\mathbb{Q})$ in terms of \eqref{Eq_ch in K of M} as follows
\begin{align*}
    ch(\Phi_{K}((f^*\gamma_{\mathbb{H}}))\tilde{z})&= \Phi_H(f^*\gamma_\mathbb{H})(1+ \frac{\psi_k}{12} +\dots )z= \Phi_H(f^*\gamma_\mathbb{H})z,
\\
     ch(\Phi_{K}((f^*\gamma_{\mathbb{H}}))\tilde{\psi}_j)&= \Phi_H(f^*\gamma_\mathbb{H})(1+ \frac{\psi_k}{12} +\dots )(\psi_j + \tau_j z)\\&= \Phi_H(f^*\gamma_\mathbb{H})\psi_j + \tau_j \Phi_H(f^*\gamma_\mathbb{H})z,
\\
     ch((\Phi_{K}(f^*\gamma_{\mathbb{H}})(\tilde{\psi}_k-g_{k,k}\tilde{\psi}_{k-1}))&= \Phi_H(f^*\gamma_\mathbb{H})(1+ \frac{\psi_k}{12} +\dots )(\psi_k +\tau_k z- g_{k,k}\psi_{k-1}-g_{k,k}\tau_{k-1}z)\nonumber\\&= \Phi_H(f^*\gamma_\mathbb{H})(\psi_k -g_{k,k}\psi_{k-1})+ \Phi_H(f^*\gamma_\mathbb{H})(\tau_k - g_{k,k}\tau_{k-1})z.
\end{align*} 
\end{proof}

We now turn our attention to the attaching map $$S^{10} \stackrel{L(E)}{\to}  (S^4 \vee S^7)^{\vee k-1}\to{(S^7)}^{\vee k-1}.$$ In order to identify the composite we are required to compute the $KO$-theoretic $e$-invariant.
 We know that
\begin{align*}
 KO^*= \mathbb{Z}[\eta, u][\mu^{\pm 1}]/(2\eta, \eta^3,\eta u, u^2-4\mu) \quad \text{and} \quad
  K^*= \mathbb{Z}[\beta^{\pm 1}], 
\end{align*}
with $|\eta| = -1,~ | u|= -4,~ |\mu| =-8$ and $|\beta| = -2$.
 The complexification map $c\colon KO\rightarrow K$
 induces a graded ring homomorphism $c\colon KO^*(X) \rightarrow K^*(X)$ with $$\label{complex}
     c(\eta)=0, \quad c(u)= 2\beta^2, \quad c(\mu)= \beta^4.$$

\begin{theorem}\label{Th_divisibility}
 Let $S^3\rightarrow E \rightarrow M_k$ be the principal $SU(2)$-fibration classified by a given  map $f\colon M_k\rightarrow \mathbb{HP^{\infty}}$. Suppose that $\Sigma L(M_k) \equiv 0\pmod{\lambda}$. Then $\Sigma (r_j\circ L(E)) \equiv 0\pmod{\lambda}, 1\leq j\leq k-1$ where $r_j:(S^4 \vee S^7)^{\vee k-1} \to S^7$ is the retraction onto the $j$-th factor.
\end{theorem}

\begin{proof}
   We consider the Atiyah-Hirzebruch spectral sequences for $\HH P^{\infty}$ and $M_k$ 
   \begin{align*}
E^{*,*}_2&=  H^{*}(\mathbb{H}P^{\infty};\mathbb{Z})\otimes KO^*(pt)\implies KO^*(\mathbb{HP^{\infty}}),\\  
E^{*,*}_2&=  H^{*}(M_k;\mathbb{Z})\otimes KO^*(pt)\implies KO^*(M_k).
   \end{align*}
The spectral sequences have no non-trivial differentials for degree reasons. 
Thus, $KO^*(\mathbb{HP^\infty})\cong KO^*[\hat{y}],$ where $\hat{y} \in KO^4(\HH P^\infty)$. The class $\hat{y}$ serves as $KO$-theoretic Thom class for $\gamma_{\HH}$.
We have $$c (\Phi_{KO}(\gamma_{\mathbb{H}}))= c(\hat{y})= \beta^{-2}\Phi_{K}(\gamma_{\mathbb{H}}).$$
Let $\hat{\psi}_{j}\in \tilde{KO}^{4}(M_k)$, $\hat{z} \in \tilde{KO}^{8}(M_k)$ be the class in the $E^{\infty}$-page represented by $\psi_j \in H^4(M_k)$ and $z \in H^8(M_k)$. Then we get
$$c(\hat{\psi}_{i})= \beta^{-2}\tilde{\psi}_{i}, \quad c(\hat{z})= \beta^{-4}\tilde{z} \quad \text{and} \quad \hat{\psi}_{i}\hat{\psi}_{j}=g_{i,j}\hat{z}. $$
It follows that the $K$-theoretic generators $\Phi_K(f^* \gamma_{\HH})\tilde{\psi}_j$ and $\Phi_K(f^* \gamma_{\HH})(\tilde{\psi}_k-g_{k,k}\tilde{\psi}_{k-1}) $ lie in the image of the map $c$. Therefore by \eqref{Eq_Adam's result}, we get in Case 2 that
   \[ e(r_{j}\circ L(E))\equiv 
    \begin{cases}
\frac{g_{j,j}-2l_j}{24} 
    &\pmod{\mathbb{Z}} , \quad\text{for } j<k-1,\\
    \frac{(g_{kk}-2l_k)-g_{k,k}(g_{k-1,k-1}-2l_{k-1})}{24} &\pmod{\mathbb{Z}} ,\quad\text{for }   j=k-1;
  \end{cases}\]
  and in Case 1 that $$e(r_{j}\circ L(E))\equiv 
\frac{g_{j,j}-2l_j}{24} 
    \pmod{\mathbb{Z}} , \quad\text{for } j\leq k-1.$$
The result now follows from Proposition \ref{KMk}.
\end{proof}

\section{$SU(2)$-bundles over even complexes}\label{evencase}
We study the homotopy type of $E \in \PP\DD_{4,7}^{11}$, the total space of a stable principal $SU(2)$-bundle over $M_k \in \PP\DD^8_3$. In this section, we consider the case where the intersection pairing $\langle-,-\rangle \colon H_{4}(M_k)\times H_{4}(M_k) \rightarrow \mathbb{Z}$ is even i.e. $\langle x,x \rangle \in 2\mathbb{Z}$ for all $x\in H_{4}(M_k)$. Note that in this case, $k$ must be even. We observe the following.
\begin{enumerate}
    \item If $k \geq 4$, $M_k$ supports a principal $SU(2)$-bundle whose total space $E$ is $3$-connected.
    \item The possible stable homotopy types of $E$ can be determined directly from the stable homotopy type of $M_k$ and the intersection form. In this regard, the formulas in Theorem \ref{Th_divisibility} are used to demonstrate this connection.
\end{enumerate}

\begin{mysubsection}{Existence of $SU(2)$-bundles}
We discuss the existence of principal $SU(2)$-bundle over $M_k \in \PP\DD^8_3$ with an attaching map as in \eqref{Eq_general L(M) define} whose intersection form is even. 
If $\Rank(H_{4}(M))= 2$, then up to isomorphism
$\begin{pmatrix}
  0 & 1 \\ 
  1 & 0  
\end{pmatrix}$
is the only possible matrix for the intersection form.
The attaching map $L(M)\in \pi_{7}(S^{4}\vee S^{4})$ of $M$ is of the form 
\begin{myeq}\label{Eq_rank 2 LM}
L(M)= [\alpha_{1},\alpha_{2}]+ l_{1}v'_{1}+ l_{2}v'_{2},\end{myeq} where $l_{1}, l_{2}\in \mathbb{Z}/12$.
For a principal bundle $SU(2) \to E \to M_k$ where $E$ is $3$-connected, the classifying map $f_E \colon M_k\rightarrow \mathbb{H}P^{\infty}$ is such that $\pi_{4}(f)$ is surjective. 
Suppose $\alpha_{i}\mapsto n_{i}\iota_{4}, i=1,2$ such that $\gcd(n_{1},n_{2})=1$ and $f\circ L(M)\simeq *$.
If both $l_{1}$ and $l_{2}$ are odd, no such $n_1$ and $n_2$ exists.
However, for $k \geq 4$, one may always construct a suitable map. 

\begin{prop}\label{existence of classifying map}
    Suppose $M_k\in \mathcal{PD}_{3}^{8}$  such that $\Rank(H_{4}(M_k))=k\geq 4$ and the intersection form is even. Then there exists a map $\psi \colon M_k\rightarrow \mathbb{H}P^{\infty}$ such that $\mbox{hofib}(\psi)$ is $3$-connected.
\end{prop}

\begin{proof}
From \cite[Theorem 4.20]{BaGh2023}, if $k\geq 6$, the attaching map of $M_k$ can be expressed as $$L(M_k)= \sum_{1\leq i < j\leq k}g_{i,j}[\alpha_{i},\alpha_{j}]+ \sum_{1=1}^{k}\frac{g_{i,i}}{2}[\alpha_{i},\alpha_{i}]+ \sum_{i=1}^{k}s_{i}\nu_{i}'$$ such that $s_{i}=0$ for $i\geq 2$ for a choice of basis $\{\alpha_{1}, \dots, \alpha_{k}\}$ of $H_{4}(M_k)$. 
Then the map $\tilde{\psi}\colon (S^{4})^{\vee k}\xrightarrow{(0,0,\dots,0, 1)} \mathbb{H}P^{\infty}$ extends to a map $\psi\colon M_k\rightarrow \mathbb{H}P^{\infty}$ such that $\pi_{4}(\psi)$ is surjective. 

 Now if $k=4,$ by \cite{MiHu73}, the attaching map of $M_k$ can be expressed as $$L(M_k)= [\alpha_{1},\alpha_{2}]+ [\alpha_{3},\alpha_{4}]+ \sum_{i=1}^{4}l_{i}\nu_{i}'$$ for a choice of basis of $H_{4}(M_k).$ Choose two positive integers $m, n$ such that $\gcd(m,n)=1$ and $ml_{1}+ nl_{3} \equiv 0 \pmod{12}$. 
 Then the map $\tilde{\psi}\colon (S^{4})^{\vee 4}\xrightarrow{(m,0,n,0)} \mathbb{H}P^{\infty}$ extends to a map $\psi\colon M_k\rightarrow \mathbb{H}P^{\infty}$ such that $\pi_{4}(\psi)$ is surjective. 
 \end{proof}


We now focus on the stable homotopy type of the total space $E(f)$ for $f \colon M_k \to \HH P^{\infty}$ such that $\pi_4(f)$ is injective.
From the attaching map of $M_k$ as in \eqref{Eq_general L(M) define} and even intersection form, we have 
\begin{myeq}\label{Eq_stable sigma of even M}
\Sigma^{\infty}M_k\simeq \Sigma^{\infty} (S^4)^{\vee k-1}\vee\Sigma^{\infty}(Cone({\sigma(M_k)}\nu_{(7)})
\end{myeq} for some even $\sigma(M_k)$. Hence the stable homotopy type of $M_k$ is determined by $\sigma(M_k)$.

\begin{prop}\label{Prop_Stable total space E}
    Let $E(f_{\psi})$ be the total space of a principal $SU(2)$-bundle over $M_{k}\in \mathcal{PD}^{8}_{3}$, classified by a map $f_\psi\colon M_k\rightarrow \mathbb{H}P^{\infty}$ for  $k\geq 4.$ Then $$
        \Sigma^{\infty}E(f_{\psi}) \simeq \Sigma^{\infty}(S^4\vee S^7)^{\vee {k-2}}\vee \Sigma^{\infty}Cone({{\lambda (\psi) }}\Sigma^4 \nu_{(7)} ),$$ where $\lambda(\psi) := \lambda_s(E(f_{\psi}))$, is even and a multiple of $\sigma(M_k)$.
 \end{prop}
 \begin{proof}
     Note that $E(f_{\psi}) \in \mathcal{PD}_{4,7}^{11}$ and its stable homotopy type is given in the Corollary \ref{Cor_stable lambda and epsilon} where $\epsilon= \epsilon_s(E(f_\psi))\in \mathbb{Z}/2$ and $\lambda(\psi)\in\mathbb{Z}/24$. So, it suffices to show that $2|\lambda(\psi)$ and $\epsilon=0$. 
     The fact $2|\lambda(\psi)$ follows from \eqref{Eq_stable sigma of even M} and Theorem \ref{Th_divisibility}. 
     
     Now, the cofibre sequence obtained from cell structure of $E$ and $M_k$ induces 
 the following commutative diagram    
 $$  \xymatrix@R+1pc@C+3pc{
\pi_{10}^{s}(S^{10})\ar[r]^{\phi_*} &  \pi_{10}^s(S^4\vee S^7)^{\bigoplus {k-1}}\ar[d]{}\ar[r]&{}\pi_{10}^{s}(E)\ar[d]{}\\
\pi_{10}^{s}(S^{7})\ar[r]_{0}^{\pi_{10}^s(\sum^{\infty} L(M))}&  \pi^s_{10}(S^4)^{\bigoplus k}\ar[r]&{}\pi_{10}^{s}(M_k)
   }$$
  of stable homotopy groups where $\phi$ is the attaching map of top cell in $E$.

Since $\lambda(\psi) \beta_{k-1} \circ \nu_{7} + \epsilon \alpha_{k-1} \circ x=0$ in $\pi^s_{10} (E)$, its image in $\pi_{10}^s(M)$ is $0$.   
 Note that bottom left map $\pi_{10}^s(\sum^{\infty} L(M))=0$ because $2|\sigma(M_k)$ in \eqref{Eq_stable sigma of even M} and hence the bottom right map is injective, where $\pi_{10}^s(S^4)^{\oplus k}\cong \mathbb{Z}/2\{\alpha_1\circ \nu^2,\dots,\alpha_{k}\circ\nu^2\}$.
 Since $2|\lambda(\psi)$, the middle vertical arrow sends $\Sigma^{\infty}\phi= \lambda(\psi) \beta_{k-1}\circ\nu_{(7)}+\epsilon \alpha_{k-1}\circ\nu^2$  to $\epsilon\alpha_{k-1}\circ\nu^2$ which is in turn mapped to $0$ via the bottom right map as $\Sigma^{\infty}\phi=0\in\pi_{10}^s(E).$ Hence $\epsilon=0\in\mathbb{Z}/2$.
 \end{proof}

\end{mysubsection}

\begin{mysubsection}{Stably trivial manifolds} The following result states that the total space of a principal $SU(2)$-bundle over stably trivial $M_{k}\in \mathcal{PD}^{8}_{3}$ ( i.e., $\sigma(M_k)\equiv 0\pmod{24}$), is itself a connected sum of copies of $S^4\times S^7$.
\begin{prop}\label{stabtriv}
   Let $E\in \PP\DD^{11}_{4,7}$  be the total space of a principal $SU(2)$-bundle over stably trivial $M_{k}\in \mathcal{PD}^{8}_{3}$. Then $E\simeq \#^{k-1}(S^4\times S^7)$.
\end{prop}

\begin{proof}
From Theorem \ref{Th_general classification of E}, we have $E\simeq \#^{k-2}E_{0,0,0}\#E_{0,\epsilon,\delta}$.
Then the result follows from Proposition \ref{Prop_Stable total space E} for $k \geq 4$. 
For $k=2$, the attaching map of $M_2$ is of the form \eqref{Eq_rank 2 LM}. Stably trivial condition implies $M_2=S^4 \times S^4$ which further implies $E=S^4 \times S^7$. 
\end{proof}

\end{mysubsection}

\begin{mysubsection}{Possible stable homotopy types of the total space}
Let $\psi\in H^{4}(M_k;\mathbb{Z})$ be a cohomology class represented by a map $\overline{\psi}\colon M_{k}\rightarrow K(\mathbb{Z},4)$ which has a unique lift $\tilde{\psi}\colon M_k\rightarrow S^4$ up to homotopy if ${\overline{\psi}} \big{|}_{{(S^4)}^{\vee {k-1}}}\circ L(M_k) \in \pi_7(S^4)$ is $0$. As the inclusion $(S^4)^{\vee{k}}\rightarrow M_{k}$ induces an isomorphism on $H_4$ and $H^4$, a cohomology class $\psi\in H^4(M_k)$ always induces a map $\tilde{\psi}\colon (S^4)^{\vee{k}}\rightarrow S^4.$ We consider the following diagram 
\[
\xymatrix{
(S^4)^{\vee{k}}\ar[r]{}\ar[d]^{\tilde{\psi}}&M_k\ar@{-->}[d]{}\\
S^4\ar[r]{}&\mathbb{H}P^{2}
}
\]
and formulate when the map $M_{k}\rightarrow \mathbb{H}P^2$ exists. Note that if the map exists then its homotopy fibre will be $3$-connected.

For $\psi\in H^4(M_k),$ consider the composite $$   S^7\xrightarrow{L(M_k)}(S^4)^{\vee{k}}\xrightarrow{\tilde{\psi}}S^4$$
and define $\tau(\psi)=[\tilde{\psi}\circ L(M)]\in \pi_{7}^s(S^4)$. Thus $$
\tau\colon H^4(M_k)\rightarrow \mathbb{Z}/24$$
and one can check it is a linear map. 

\begin{prop}\label{Prop_tau map}
Suppose $\psi \in H^4(M_k)$ is primitive. The map $\tilde{\psi}\colon (S^4)^{\vee{k}}\rightarrow S^4$ extends to a map $f_\psi: M_k\rightarrow \mathbb{H}P^2$ if and only if $\psi\cup\psi\equiv \tau(\psi)z\pmod{24}$ for some $z \in H^8(M_k)$.
\end{prop}

\begin{proof}
Consider a primitive element $\psi \in H^4(M_k)$. We extend $\psi$ to a basis of $H^{4}(M_k)$, and use the dual basis of $\pi_4(M_k)$ to write down the attaching map of $M_{k}$ is as in $\eqref{Eq_general L(M) define}$. 
In this notation, we have $\psi^2=g_{k,k} z$ where $z \in H^8(M_k)$ is the chosen generator and $\tau (\psi)=g_{k,k}-2l_k$. 
Thus $\tilde{\psi}\circ L(M)$ maps to $0\in \pi_{7}(\mathbb{H}P^{2})$ if and only if $l_k=0$, that is $\psi\cup \psi\equiv \tau{(\psi)}z\pmod{24}$. Hence, the result follows.
\end{proof}

Proposition \ref{Prop_tau map} gives us criteria for constructing the maps $f_\psi$ out of cohomology classes $\psi$. Using this we determine which multiples of $\sigma(M_k)$ may occur as $\lambda_s(E)$ for $E\to M_k$ a principal $SU(2)$-bundle where $E$ is $3$-connected. We first show that there exist $f_\psi$ such that $\lambda(\psi)=\sigma(M_k)$ if $k$ is large enough.
\begin{prop}\label{Prop_sigma k achieved}
Suppose the stable homotopy type of $M_{k}$ is determined by $\sigma(M_k).$ Then there exists $f_\psi\colon M_{k}\rightarrow \mathbb{H}P^{\infty}$ such that
\begin{enumerate}
    \item $\lambda(\psi)\equiv \sigma(M_k) \pmod{3}$ for $k \geq 5$,
    \item $\lambda(\psi) \equiv \sigma(M_k) \pmod{8}$ for $k \geq 7$.
\end{enumerate}
\end{prop}

\begin{proof}

 We begin the proof with the first case. If $\tau \equiv 0 \pmod{3}$, the proof follows from Theorem \ref{Th_divisibility}. Let $\tau \not\equiv 0 \pmod{3}$ and $\psi_{0}$ a primitive cohomology class such that $\tau(\psi_{0}) \equiv \sigma(M_k) \not \equiv 0 \pmod{3}$. 
 We need to choose a $\psi \in \ker (-\cup\psi_{0})\cap \ker(\tau)$ such that  $\psi^2 \equiv \tau(\psi)\equiv 0\pmod{3}$. This we can do for $k-2 \geq 3$, see \cite[Chapter II, (3.2)-(3.4)]{MiHu73}.
By Poincar\'e duality, we get $\psi'$ such that $\psi\cup \psi'=z$. We write  $\psi'= \psi'_{\ker (\tau)}+t\psi_{0}$ for some $t$ where $\psi'_{\ker(\tau)}\in \ker(\tau)$. 
Since $z=\psi'\cup \psi= \psi'_{\ker(\tau)} \cup \psi$, we may assume $\psi' =\psi'_{\ker\tau}$.
Thus by assigning $\tau_1=\tau(\psi_{0}) \equiv \sigma(M_k) \pmod{3}$, $\tau_{k-1}=\tau(\psi') \equiv 0 \pmod{3}$ and $\tau_{k}=\tau(\psi)\equiv 0 \pmod{3}$; and choosing other $\psi_i$'s such that $\tau_i=0$ for $i=2, \dots, k-2$ we have $$\lambda(\psi)\equiv \sigma(M_k) \pmod{3} \quad \quad \text{for } k \geq 5.$$


Now we look into the second case. The proof goes similarly to that of the above, except when we choose $\psi \in \ker (-\cup\psi_{0})\cap \ker(\tau)$ such that  $\psi^2 \equiv \tau(\psi)\equiv 0\pmod{8}$. We can choose such $\psi$ for $k-2 \geq 5$, see \cite[Chapter II, (3.2)-(3.4)]{MiHu73}. Then following similar arguments one can deduce $$\lambda(\psi)\equiv \sigma(M_k) \pmod{8} \quad \quad \text{for } k \geq 7.$$
\end{proof}

The following theorem constructs $f_\psi$ with $\lambda(\psi)=3 \sigma(M_k)$ if $\sigma(M_k)$ is not divisible by $3$. 

\begin{theorem}\label{Theorem_divisibility theorem for 3} Suppose $3 \nmid \sigma(M_k)$. Then for $k\geq 7$, there exists $f_\psi\colon M_k\rightarrow \mathbb{H}P^{2}$ such that $\lambda(\psi)=3 \sigma(M_k)$.
\end{theorem}

\begin{proof}
    Let $$\tau^{(3)}\colon H^{4}(M_k,\mathbb{F}_{3})\rightarrow \mathbb{F}_3$$ be the restriction of $\tau$ in modulo $3$. 
    As $\tau^{3}$ is surjective, there exists $\psi_{0}\in H^{4}(M_k,\mathbb{F}_3)$ such that for all cohomology class $\psi$, $\tau^{(3)}(\psi)z\equiv\psi\cup \psi_0\pmod{3}$. In particular, $\tau^{(3)}(\psi_0)z\equiv \psi_0\cup\psi_0\pmod{3}$. 
    We consider two cases, $\psi_0^2=0$ or $\psi_0^2$ is unit. 
    
    First let $\psi_0^2$ is unit.
    Then we can choose $\psi_1, \dots, \psi_{k-1}$ such that $\psi_0 \cup \psi_i=0$. 
    We take the dual basis $\alpha_i$ corresponding to $\psi_i$ and $\alpha_k$ corresponding to $\psi_0$. Thus $\psi_0(\alpha_i)=0$ for $i=1, \dots, k-1$ and $\psi_0(\alpha_{k})=1$. Hence $$\tau_1 \equiv \dots \equiv \tau_{k-1} \equiv 0 \pmod{3}, \quad \text{and} \quad \tau_k \equiv g_{k,k} \equiv 1 \pmod{3},$$ which implies $\lambda(\psi)= \gcd(\tau_1,\dots,\tau_{k-2},\tau_{k-1}) \equiv 0 \pmod{3}$.
    
    Now let $\psi_0^2=0$. 
    Then $\tau^{(3)}(\psi_0)z=0$. We choose $\psi_1, \dots, \psi_{k-1}$ such that $\psi_{k-1} \cup \psi_0=1$ and $\psi_i \cup \psi_0=0$ for $i=1 , \dots, k-2$. After taking the dual basis, with similar argument we have $$\tau_1 \equiv \dots \equiv \tau_{k-2}\equiv 0 \pmod{3}, \quad \tau_{k-1}\equiv \sigma(M_k) \pmod{3}, \quad \text{and} \quad \tau_k \equiv g_{k,k} \equiv 0 \pmod{3}.$$
    Hence $\lambda(\psi)= \gcd(\tau_1,\dots,\tau_{k-2},\tau_k-g_{k,k}\tau_{k-1})\equiv0\pmod{3}$.     
\end{proof}

\begin{rmk}
We note that Proposition \ref{Prop_tau map}, Proposition \ref{Prop_sigma k achieved}, and Theorem \ref{Theorem_divisibility theorem for 3} does not use the fact that the intersection form is even, and also holds in the case where the intersection form is odd. 
\end{rmk}

Now we look to prove similar results  modulo $8$, which in turn provide us desired construction $\psi$ as in Proposition \ref{Prop_tau map} using the Chinese reminder theorem. However, in this case certain conditions are required for obtaining analogous $f_\psi$. 

\begin{defn}
\begin{itemize}
\item    A complex $M_k$ with $\sigma(M_k)=2 $ or $4$, is said to satisfy hypothesis $(H_8)$ if $(\ker \tau)^{\perp}=(\sigma(M_k) \psi) \pmod{8}$ where $\psi \in H^4(M_k)$ (which is unique $\pmod{\frac{8}{\sigma(M_k)}}$) satisfies
    $$\begin{cases} \psi ^2 \equiv  0 \pmod{8} & \text{ if } \sigma(M_k)=2 \\
      \psi ^2 \equiv  0 \pmod{4} & \text{ if } \sigma(M_k)=4.\end{cases}$$

\item    A complex  $M_k$ with $\sigma(M_k)=2 $ is said to satisfy hypothesis $(H_4)$ if $(\ker \tau)^{\perp}=(2\psi) \pmod{4} $ where $\psi \in H^4(M_k)$ (which is unique $\pmod{2}$) satisfies
    $$\psi ^2 \equiv \tau (\psi) \equiv 0 \mbox{ or } 4 \pmod{8}.$$
\end{itemize}
\end{defn}

Note that the hypotheses $(H_8)$ and $(H_4)$ depends only on the intersection form and $\tau$ and not on the choice of $\psi$. We now prove the existence of $f_\psi$ under the hypothesis defined above.
\begin{theorem} \label{divfor4}
 \begin{enumerate}  
\item  Suppose $8 \nmid \sigma(M_k)$. For $k\geq 5$, there exist $f_\psi \colon M_k\rightarrow \mathbb{H}P^{2}$ such that $\lambda(\psi)\equiv 0 \pmod{8}$ if and only if the complex satisfies hypothesis $(H_8)$.

\item Suppose $ \sigma(M_k)=2$. Then for $k\geq 5$, there exist $f_\psi\colon M_k\rightarrow \mathbb{H}P^{2}$ such that $\lambda(\psi)\equiv 4 \pmod{8} $ if and only if the complex satisfies hypothesis $(H_4)$.

\end{enumerate}
\end{theorem}

\begin{proof}
   The condition $k\geq 5$ comes from the fact that we are required to make certain choices modulo $3$ using Proposition \ref{Prop_sigma k achieved}. First suppose in case (1), a $f_\psi$ exists such that $\lambda(\psi)=8 \sigma(M_k)$. Then there exists a basis $\{\psi_1, \dots, \psi_{k-2}, \psi', \psi\}$ satisfying (working $\pmod{8}$)
 \begin{myeq}\label{cond8}
\begin{aligned} &\psi \cup \psi_i=0 \text{ for }1 \leq i \leq k-2, \quad &&\psi \cup \psi'=1, \\ &\tau(\psi_i)=0 \text{ for }1 \leq  i \leq k-2, \quad &&\tau (\psi') = \sigma(M_k) \quad \text{and} \quad \tau(\psi)=\psi^2=0.\end{aligned}
\end{myeq}
    Note that $\langle \psi \rangle^{\perp}= \langle \psi_1, \dots, \psi_{k-2}, \psi \rangle \subset \ker (\tau)= \langle \psi_1, \dots, \psi_{k-2}, \psi, \frac{8}{\sigma(M_k)} \psi' \rangle$. This implies $(\ker (\tau))^{\perp} = (\sigma(M_k) \psi)$, and thus the hypothesis $(H_8)$ is satisfied. 

 For the converse part if the complex satisfies hypothesis $(H_8)$, one can check that there is a choice of $\psi$ such that  \eqref{cond8} is satisfied. We look into the cases $\sigma(M_k)=2, 4$.

    First, let $\sigma(M_k)=2$. Then $(\ker (\tau))^{\perp} = (2\psi)$ and $\psi$ is well defined modulo $4$. We note that 
\begin{myeq}\label{psiprod2}
 \chi \cup (2\psi) \equiv \tau(\chi)z \pmod{8} ~\forall \chi \in H^4(M_k).
\end{myeq} 
This implies 
$$\tau (\psi )\equiv 2\psi^2 \equiv 0 \pmod{8}.$$
    For any $\chi \in \ker (\tau)$, we have $(2 \psi) \chi =\tau (\chi) \equiv 0 \pmod{8}$. In particular $2 \psi^2 = \tau(\psi) \equiv 0 \pmod{8}$. Together these two implies $\psi^2 \equiv 0 \pmod{8}$. The equation \eqref{psiprod2} implies that if $\chi\cup \psi =0$, $\tau(\chi)=0$. Now choosing a basis as in Case (2) of Proposition \ref{Prop_tau j calculate}, we obtain the conditions in \eqref{cond8}.
    
    Now let $\sigma(M_k)=4$. Then $(\ker (\tau))^{\perp} = (4\psi)$ and $\psi$ is determined modulo $2$. We proceed analogously observing that 
\[
 \chi \cup (4\psi) \equiv \tau(\chi)z \pmod{8} ~\forall \chi \in H^4(M_k),
\] 
which implies 
$$\tau (\psi )\equiv 4\psi^2 \equiv 0 \pmod{8}.$$
Let $\psi'$ be such that $\tau(\psi')=4$, and so we have that $\psi \cup \psi'$ is an odd multiple of $z$. If $\psi^2$ is $4 \pmod{8}$ we change $\psi$ to $\psi + 2 \psi'$ to ensure $\psi^2 = \tau(\psi) \equiv 0 \pmod{8}$.  Now choosing a basis as in Case (2) of Proposition \ref{Prop_tau j calculate}, we obtain the conditions in \eqref{cond8}.

  The case (2) also proceeds along analogous lines.   For the existence (of $f_\psi$ for some $\psi$) question, we need a basis satisfying 
\begin{myeq}\label{cond4}
\begin{aligned} &\psi \cup \psi_i=0 \text{ for }1 \leq i \leq k-2, \quad \psi \cup \psi'=1, \\ 
&\tau(\psi_i)\equiv 0 \pmod{4} \text{ for }1 \leq  i \leq k-2, \quad 
\tau (\psi') = \sigma(M_k),  \quad \tau(\psi)=\psi^2 \equiv 0 \mbox{ or } 4 \pmod{8}, \\
&\mbox{ such that at least one of } \tau(\psi_i) \mbox{ for } 1\leq i \leq k-2, \mbox{ or } \tau(\psi) \equiv 4 \pmod{8} .\end{aligned}
\end{myeq}
    Note that $\langle \psi \rangle^{\perp}= \langle \psi_1, \dots, \psi_{k-2},\psi \rangle \subset \ker (\tau)= \langle \psi_1, \dots, \psi_{k-2},\psi, 2 \psi' \rangle \pmod{4}$ and $(\ker (\tau))^{\perp} = (2 \psi)\pmod{4}$. Hence, the hypothesis $(H_4)$ is satisfied. 

Conversely, if $(H_4)$ is satisfied, we obtain a $\psi$ such that $\psi^2 \equiv \tau(\psi) \equiv 0 \mbox{ or } 4 \pmod{8}$. This $\psi$ also satisfies 
\[
 \chi \cup (2\psi) \equiv \tau(\chi)z \pmod{4} ~\forall \chi \in H^4(M_k).
\] 
Let $\psi'$ be such that $\tau(\psi')=2$, which implies that $\psi \cup \psi'$ is an odd multiple of $z$. Now replacing $\psi$ by $\psi + 2 \psi'$ if required we may assume that $\psi^2 \equiv \tau(\psi) \equiv  4 \pmod{8}$.  Now choosing a basis as in Case (2) of Proposition \ref{Prop_tau j calculate}, we obtain the conditions in \eqref{cond4}.
\end{proof}

If $\Rank(H_4(M_k))=k \geq 5$, the above results indicate a systematic computation of possible stable homotopy types of the total space depending on $k$, $\sigma(M_k)$ and the intersection form. In lower rank cases, the results depend on the explicit formula for the attachment $L(M_k)$, and not just on these variables. Hence the systematic description turns out to be cumbersome. We demonstrate some observations on the $\Rank(H_4(M)) =2$ case.

\begin{exam}
Recall that the attaching map $L(M)\in \pi_{7}(S^{4}\vee S^{4})$ of $M_2$ is of the form 
$$L(M)= [\alpha_{1},\alpha_{2}]+ l_{1}v'_{1}+ l_{2}v'_{2}$$ where $l_{1}, l_{2}\in \mathbb{Z}/12$.
We already have 
\begin{itemize}
\item If both $l_{1}$ and $l_{2}$ are odd, there does not exist $f \colon M\rightarrow \mathbb{H}P^{\infty}$ such that  $E=\mbox{hofib}(f)$ is $3$-connected. 
\end{itemize}
If one of $l_1$ and $l_2$ is even, or both are even, there exists $f \colon M\rightarrow \mathbb{H}P^{\infty}$ such that  $E=\mbox{hofib}(f)$ is $3$-connected. Via an explicit calculation using Proposition \ref{Prop_tau j calculate}, we observe the following. 
\begin{enumerate}
    \item If none of $l_1$ and $l_2$ are divisible by $3$, then we obtain $\lambda(E) \equiv 0 \pmod 3$.

    \item If either of $l_1$ or $l_2$ but not both are divisible by $3$, then  $\lambda(E) \not \equiv 0 \pmod 3$.

    \item If $\sigma(M) \equiv 4 \pmod{8}$ and $ l_1 l_2 \equiv 0 \pmod{8}$ where none of $l_1$ and $l_2$ are divisible by $3$, then $\lambda(E) \equiv 0 \pmod{8}$. 

    \item If $\sigma(M) \equiv 2 \pmod{8}$, we can never obtain $\lambda(E) \equiv 0, 4 \pmod{8}$. 
\end{enumerate}
\end{exam}
\end{mysubsection}

\section{$SU(2)$-bundles over odd complexes}\label{oddcase}

We now work out the case of $M_k \in \PP\DD_3^8$ for which the intersection form is odd. Recall that the notation $M_k$ means that $\Rank(H_4(M_k))=k$. The intersection form being odd implies that there are two possibilities of $\sigma(M_k)$, namely, $1$ and $3$ among the divisors of $24$. Here, we prove \\
1) For $k\geq 3$, it is possible to obtain a $SU(2)$-bundle whose total space is $3$-connected. \\ 
2) Further if $k\geq 7$, it is possible to obtain maps $\psi(j) : M_k \to \HH P^\infty$ with $\lambda(\psi(j))=j$ for every multiple $j \pmod{24} $ of $\sigma(M_k)$ which is also a divisor of $24$. 

\begin{mysubsection}{Existence of $SU(2)$-bundles}
Through an explicit computation, we demonstrate the existence of principal $SU(2)$-bundles
over $M_k \in \PP \DD ^8_3$ if $k\geq 3$ when the intersection form is odd.

\begin{prop}\label{Prop_existence of classifying map odd case}
    Suppose $M_k\in \mathcal{PD}_{3}^{8}$  such that $\Rank(H_{4}(M_k))=k\geq 3$ and the intersection form is odd. Then there exists a map $\psi \colon M_k\rightarrow \mathbb{H}P^{\infty}$ such that $\mbox{hofib}(\psi)$ is $3$-connected.
\end{prop}

\begin{proof}
    Recall that the attaching map of $M_k$ can be expressed as \eqref{Eq_general L(M) define}. Using Proposition \ref{Prop_tau map}, we are required to find a primitive element $\psi$ such that $\tau(\psi)z \equiv \psi^2 \pmod{24}$, which is equivalent to checking that the coefficient of $\nu'$ in $\tilde{\psi}\circ L(M_k)$ is $0 \pmod{12}$. It suffices to find $\psi \pmod{8}$ and $\psi \pmod{3}$ separately. 

We first work out the $\pmod{3}$ case, where the base ring is a field of characteristic $\neq 2$ so that the form is diagonalizable. Considering the map 
$(S^4)^{\vee k}\xrightarrow[]{(0,\dots,0, n_1, n_2, n_3)} \HH P^{\infty}$ which sends 
 $L(M_k)$ 
to 
    \[
 \Big(\pm{n_1 \choose 2} \pm{n_2 \choose 2} \pm{n_3 \choose 2} + n_1 l_{k-2}+n_2 l_{k-1}+ n_3 l_k\Big) \cdot \nu' + \mbox{ multiple of } \nu,
    \]
where  the $\pm$ correspond to the diagonal entries $\pmod{3}$. We observe through a direct calculation that for every fixed choice $\epsilon_1, \epsilon_2,\epsilon_3$ of $\pm 1$s, there is a $n_1, n_2, n_3$ with $\gcd(n_1, n_2, n_3)=1$ such that 
\[
 \Big(\epsilon_1{n_1 \choose 2} +\epsilon_2{n_2 \choose 2} +\epsilon_3{n_3 \choose 2} + n_1 l_{k-2}+n_2 l_{k-1}+ n_3 l_k\Big) \equiv 0 \pmod{3}.
   \]
This completes the argument $\pmod{3}$.

 Working $\pmod{8}$, the fact that the intersection form is odd implies that we may choose a basis such that $g_{k,k} \equiv \pm 1 \pmod{8}$ and $g_{k, k-1} \equiv 0 \pmod{8}$, see \cite[Chapter II, (4.3)]{MiHu73}.  
    Also the intersection form can be written as the block matrix $\begin{pmatrix}
        A & 0 \\0 & \pm 1
    \end{pmatrix}.$ 
    
    If $A$ is an even intersection form, then the result follows from Proposition \ref{existence of classifying map} for $k \geq 5$.
    Now let $A$ is not even, then the intersection form is $\begin{pmatrix}
        A' & 0 \\ 0 & B'
    \end{pmatrix}$ where $B'$ is a diagonal matrix of order $2$ with diagonal entries $\pm 1$. If $A'$ is an even intersection form, then for $k \geq 6$ the result follows from Proposition \ref{existence of classifying map}.
    If $A'$ is not even, the intersection form updates to 
\[
        \begin{pmatrix}
        A'' & 0 \\0 & B''
    \end{pmatrix}
    \] 
    where $B''$ is a diagonal matrix of order $3$ with diagonal entries $\pm 1$.
For $B''=I_3$, the map $(S^4)^{\vee k}\xrightarrow[]{(0,\dots,0, n_1, n_2, n_3)} \HH P^{\infty}$ sends $L(M_k)$ 
to 
    \[
 \Big({n_1 \choose 2} +{n_2 \choose 2} +{n_3 \choose 2} + n_1 l_{k-2}+n_2 l_{k-1}+ n_3 l_k\Big) \cdot \nu' + \mbox{ multiple of } \nu,
    \]
    where $\gcd (n_1, n_2, n_3)=1$ ensures that the corresponding $\psi$ is primitive. We may directly compute and observe that the equations 
\[ \Big({n_1 \choose 2} +{n_2 \choose 2} +{n_3 \choose 2} + n_1 l_{k-2}+n_2 l_{k-1}+ n_3 l_k\Big) \equiv 0 \pmod{8}, \gcd (n_1, n_2, n_3)=1,\] 
have a common solution. A similar argument works for the other diagonal $\pm 1$ matrix choices for $B'''$. This proves the result for $k\geq 6$. 

If $3 \leq k \leq 5$, we know from \cite[Chapter II, (3.2)-(3.4)]{MiHu73} that the form is a direct sum of $\pm 1$ and the hyperbolic form. The argument in the even case  in  Proposition \ref{existence of classifying map} implies the result for a sum of two hyperbolic forms, and the above argument implies the result for a sum of $3$ $\pm 1$s. The remaining cases are taken care of if we show the result for the intersection form 
\[
        \begin{pmatrix}
            H & 0 \\ 0 & \pm 1
        \end{pmatrix}
    \]
    where $H= \begin{pmatrix}
        0 &1 \\1 &0
    \end{pmatrix}$ is the hyperbolic matrix. 
    We consider the map $(S^4)^{\vee 3} \xrightarrow[]{(n_1, n_2, n_3)} \HH P^{\infty}$, and as above we need to find a common solution of  
    $${n_3 \choose 2}+ n_1 n_2+\sum_{i=1}^3 n_i l_i \equiv 0 \pmod{8},  \gcd (n_1, n_2, n_3)=1.$$ 
Through a direct calculation, we check that such solutions always exist. This completes the proof. 
\end{proof}

The following example shows that Proposition \ref{Prop_existence of classifying map odd case} does not extend to the $k=2$ case. 
\begin{exam}\label{k2odd}
Consider $M_2$ such that 
\[L(M_2) = \nu_1 + \nu_2 + l_1 \nu_1' + l_2 \nu_2', \quad l_1\equiv l_2 \equiv 2 \pmod{3}.\]
A map $M_2 \to \HH P^\infty$ which restricts to $(n_1,n_2)$ on the $4$-skeleton sends $L(M_2)$ to 
\[
 \Big({n_1 \choose 2} +{n_2 \choose 2}  + n_1 l_1+n_2 l_2\Big) \cdot \nu' + \mbox{ multiple of } \nu. 
    \]
We may check directly that 
\[\Big({n_1 \choose 2} +{n_2 \choose 2}  + n_1 l_1+n_2 l_2\Big) \equiv 0 \pmod{3} \implies n_1 \equiv n_2 \equiv 0 \pmod{3}.\]
Therefore, there is no map $M_2 \to \HH P^\infty$ whose homotopy fibre is $3$-connected.
\end{exam}

\end{mysubsection}

\begin{mysubsection}{Possible stable homotopy type of the total space}
We note from \S \ref{evencase} that Propositions \ref{Prop_tau map}, \ref{Prop_sigma k achieved}, and Theorem \ref{Theorem_divisibility theorem for 3} are also valid when the intersection form is odd. We check that all stable homotopy types are achievable in the odd case if the rank $k$ of $H_4(M)$ is $\geq 7$. Applying the results from \S \ref{evencase}, it only remains to check that the different possibilities $\pmod{8}$ are achievable. We do this in the theorem below.

\begin{theorem}\label{oddcasethm}
    Suppose $\sigma(M_k)$ is odd. Then for $k \geq 5$ and $j\in \{0,2,4\}$, there exists $\psi \colon M_k \to \HH P^2$ such that $\lambda(\psi) \equiv j \pmod{8}$.
\end{theorem}

\begin{proof}
We work $\pmod{8}$, knowing that if $k\geq 5$, Proposition \ref{Prop_sigma k achieved} allows us to make a choice of $\psi$ so that $\lambda(\psi)$ is as required $\pmod{3}$. The proof is very similar to the proof of Theorem \ref{Theorem_divisibility theorem for 3}. If $\sigma(M_k)$ is odd, the linear map $\tau \colon (\Z/8)^k \to \Z/8$ is represented by some primitive class $\psi$ (that is, $\tau (\chi)= \langle \chi, \psi \rangle \pmod{8}$, and $\Z\{\psi\}$ is a summand of $H^4(M_k)$). In particular, $\psi^2=\tau (\psi)$. Now, we have two cases.

First let $\tau(\psi)$ be odd, i.e. $\psi^2$ is unit in modulo $8$. Then we can extend $\psi$ to a basis $\psi, \psi_1, \dots, \psi_{k-1}$ such that $\psi \cup \psi_i=0$ for $1\leq i \leq k-1$. 
We take the dual basis $\alpha_i$ corresponding to $\psi_i$ and $\alpha_k$ corresponding to $\psi$. Thus $\psi(\alpha_i)=0$ for $i=1, \dots, k-1$ and $\psi(\alpha_{k})=1$. Hence 
$$\tau_1 \equiv \dots \equiv \tau_{k-1} \equiv 0 \pmod{8}, \quad \text{and} \quad \tau_k \equiv g_{k,k} \pmod{8},$$ 
which implies $\lambda(\psi)= \gcd(\tau_1,\dots,\tau_{k-2},\tau_{k-1}) \equiv 0 \pmod{8}$ by Case (1) of Proposition \ref{Prop_tau j calculate}.
    
Now let $\psi^2$ be even. Extend $\psi$ to a basis $\psi, \psi_1, \dots, \psi_{k-1}$ such that $\psi_{k-1} \cup \psi=1$ and $\psi_i \cup \psi=0$ for $i=1 , \dots, k-2$. After taking the dual basis, with similar argument we have $$\tau_1 \equiv \dots \equiv \tau_{k-2}\equiv 0 \pmod{8}, \quad \tau_{k-1}\equiv 1 \pmod{8}, \quad \text{and} \quad \tau_k \equiv g_{k,k} \pmod{8}.$$
Hence $\lambda(\psi)= \gcd(\tau_1,\dots,\tau_{k-2},\tau_k-g_{k,k}\tau_{k-1})\equiv0\pmod{8}$ by Case (2) of Proposition \ref{Prop_tau j calculate}. This proves the result for $j=0$.

For $j=2, $ or $4$, we can use $\psi(j)= \psi+j \psi_1 $, and note that 
\[ \tau(\psi(j))= \tau(\psi), ~ \psi(j)^2 = \psi^2 + j^2 \psi_1^2 \equiv \psi^2 \pmod{8}.\]
The last equivalence comes from the fact that $\tau(\psi_1)z = \psi\cup \psi_1 \equiv 0 \pmod{8}$, and so $\psi_1^2$ is forced to be an even multiple of $z$. We may now compute using the formulas of Proposition \ref{Prop_tau j calculate} to conclude that $\lambda(\psi(j))\equiv j \pmod{8}$.
%
\end{proof}
    
As in the even case, when the rank is high enough we have a systematic idea of the possibilities of the total space. However, in the low rank cases ($k\leq 6$) the results are not systematic, and may depend on individual cases rather than only on $\sigma(M_k), k$, and the intersection form.

\end{mysubsection}

\end{document}